\pdfoutput=1
\documentclass[reqno, 11pt]{amsart}
\usepackage{amsmath,amsfonts,amssymb,amsthm}
\usepackage[usenames,dvipsnames,svgnames,table]{xcolor}
\usepackage{graphicx}
\usepackage[pdfauthor=Holden\ and\ Lyons,bookmarksnumbered,hyperfigures,pdftex,pagebackref,colorlinks=true,linkcolor=BrickRed,urlcolor=BrickRed,citecolor=OliveGreen,pdfstartview=FitH]{hyperref}
\usepackage{comment}
\usepackage{tabularx}
\usepackage[protrusion=true,expansion=true]{microtype}
\usepackage{enumerate}
\usepackage{bbm}
\usepackage{mathrsfs}
\usepackage[margin=1.2in]{geometry}
\usepackage{mathtools}
\mathtoolsset{centercolon}
\usepackage[final]{showlabels}

\usepackage[margin=10pt,font=small,labelfont=bf]{caption}

\numberwithin{equation}{section}

\newtheorem{theorem}{Theorem}[section]

\newtheorem{lemma}[theorem]{Lemma}
\newtheorem{proposition}[theorem]{Proposition}
\newtheorem{prop}[theorem]{Proposition} 
\newtheorem{lem}[theorem]{Lemma} 

\newtheorem{remark}[theorem]{Remark}

\newtheorem{notation}[theorem]{Notation}

\def\@rst #1 #2other{#1}
\newcommand{\N}{\mathbbm{N}}

\newcommand{\Z}{\mathbbm{Z}}
\newcommand{\R}{\mathbbm{R}}

\newcommand{\eps}{\varepsilon}
\newcommand{\1}{\mathbf{1}}

\DeclareMathOperator{\Var}{Var}

\def\cS{\mathcal{S}}

\def\cB{\mathcal{B}}

\newcommand{\aryb}{\begin{eqnarray*}}
\newcommand{\arye}{\end{eqnarray*}}
\def\alb#1\ale{\begin{align*}#1\end{align*}}
\newcommand{\eqb}{\begin{equation}}
\newcommand{\eqe}{\end{equation}}
\newcommand{\eqbn}{\begin{equation*}}
\newcommand{\eqen}{\end{equation*}}

\newcommand{\ol}{\overline}

\newcommand{\op}{\operatorname}

\newcommand{\frk}{\mathfrak}
\newcommand{\eqD}{\overset{d}{=}}

\newcommand{\wt}{\widetilde}
\newcommand{\wh}{\widehat}

\def\xt{{\wt {\bf x}}}
\def\x{{\bf x}}
\def\yt{{\wt {\bf y}}}
\def\y{{\bf y}}

\def\z{{\bf z}}

\def\w{{\bf w}}
\def\u{{\bf u}}
\def\ud{{\bf u}_{\mathrm d}}
\newcommand\p{{\mathbf p}}
\newcommand\sy[1]{\ifx#10 \alpha\else\ifx#11 \beta\else\gamma\fi\fi}

\newcommand\dtv{d_{\mathrm{TV}}}
\newcommand\dtvs{d_{\mathrm{TV}}^2}
\newcommand\dHe{d_{\mathrm H}}
\newcommand\dHes{d_{\mathrm H}^2}
\newcommand\st{\, ; \;}  %% "such that" in sets
\newcommand\dlch{\Delta}  % the law of the trace

\let\P\relax
\DeclareMathOperator{\P}{\mathbf{P}\mathopen{}}
\DeclareMathOperator{\E}{\mathbf{E}\mathopen{}}

\def\Psub_#1{\P_{\! #1}}

\def\Pbig#1{\P\mkern-.5mu\bigl[#1\bigr]}
\def\Psubbig_#1#2{\Psub_{#1}\mkern-1.5mu\bigl[#2\bigr]}

\def\Psubbigg_#1#2{\Psub_{#1}\mkern-1.5mu\biggl[#2\biggr]}
\def\PBig#1{\P\mkern-.5mu\Bigl[#1\Bigr]}

\def\Ebig#1{\E\mkern-1.5mu\bigl[#1\bigr]}
\def\Esubbig_#1#2{\E_{#1}\mkern-1.5mu\bigl[#2\bigr]}
\def\EsubBig_#1#2{\E_{#1}\mkern-1.5mu\Bigl[#2\Bigr]}
\def\Esupbig^#1#2{\E^{#1}\mkern-1.5mu\bigl[#2\bigr]}
\def\EsupBig^#1#2{\E^{#1}\mkern-1.5mu\Bigl[#2\Bigr]}
\def\Ebigg#1{\E\mkern-1.5mu\biggl[#1\biggr]}
\def\Esubbigg_#1#2{\E_{#1}\mkern-1.5mu\biggl[#2\biggr]}

\newcommand\dfn[1]{\textit{\textbf{#1}}}

\newcommand\flr[1]{\lfloor #1 \rfloor}
\newcommand\bflr[1]{\big\lfloor #1 \big\rfloor}
\newcommand\Bin{\mathrm{Bin}}
\newcommand\bern{\mathrm{Bernoulli}}

\newcommand\compl{^{\mathrm c}}  % complement
\newcommand\thmenv[3]{\begin{#1} \label{#1:#2} #3 \end{#1}}
\newcommand\richthmenv[4]{\begin{#1}[#3] \label{#1:#2} #4 \end{#1}}

\def\procl#1.#2 #3\endprocl{%
       \ifx#1t\thmenv{theorem}{#2}{#3}\fi
       \ifx#1l\thmenv{lem}{#2}{#3}\fi
       \ifx#1p\thmenv{prop}{#2}{#3}\fi
       \ifx#1c\thmenv{cor}{#2}{#3}\fi
       \ifx#1d\thmenv{defn}{#2}{#3}\fi
       \ifx#1g\thmenv{conj}{#2}{#3}\fi
       \ifx#1q\thmenv{question}{#2}{#3}\fi
       \ifx#1r\thmenv{remark}{#2}{{\rm #3}}\fi
    }%

\def\rprocl#1.#2 #3 #4\endprocl{%
       \ifx#1t\richthmenv{theorem}{#2}{#3}{#4}\fi
       \ifx#1l\richthmenv{lem}{#2}{#3}{#4}\fi
       \ifx#1p\richthmenv{prop}{#2}{#3}{#4}\fi
       \ifx#1c\richthmenv{cor}{#2}{#3}{#4}\fi
       \ifx#1d\richthmenv{defn}{#2}{#3}{#4}\fi
       \ifx#1g\richthmenv{conj}{#2}{#3}{#4}\fi
       \ifx#1q\richthmenv{question}{#2}{#3}{#4}\fi
       \ifx#1r\richthmenv{remark}{#2}{#3}{{\rm #4}}\fi
    }%

\def\rref#1.#2/{%
      \ifx #1sSection~\ref{s.#2}\fi
      \ifx #1SSubsection~\ref{S.#2}\fi
      \ifx #1tTheorem~\ref{theorem:#2}\fi
      \ifx #1lLemma~\ref{lem:#2}\fi
      \ifx #1cCorollary~\ref{cor:#2}\fi
      \ifx #1pProposition~\ref{prop:#2}\fi
      \ifx #1dDefinition~\ref{defn:#2}\fi
      \ifx #1gConjecture~\ref{conj:#2}\fi
      \ifx #1qQuestion~\ref{question:#2}\fi
      \ifx #1rRemark~\ref{remark:#2}\fi
      \ifx #1aAppendix~\ref{a.#2}\fi
      \ifx #1fFigure~\ref{f.#2}\fi
      \ifx #1e(\ref{e.#2})\fi
      \ifx #1b\citet{#2}\fi
      \ifx #1B\citep{#2}\fi
        }

\newcommand\citet{\cite}
\newcommand\citep{\cite}

\def\rlabel #1 #2{\begin{equation} \label{#1} #2 \end{equation}}

\def\rproof{\begin{proof}}

\def\eqaln#1{\begin{align*} #1 \end{align*}}

\def\beginbulletitems{\begin{itemize}}

\def\endbulletitems{\end{itemize}}

\def\Qed{\end{proof}}

\def\bsection#1#2{\bigbreak\section{#1}\label{#2}}

\DeclareMathAlphabet{\mathpzc}{OT1}{pzc}{m}{it}

\begin{document}

\author{Nina Holden}
\address{Department of Mathematics,	Massachusetts Institute of Technology} 
\email{\href{mailto:ninahold@gmail.com}{ninahold@gmail.com}}
\thanks{%
N.H.\ is partially supported by an internship at Microsoft Research and by a fellowship from the Norwegian Research Council.}

\author{Russell Lyons}
\address{Department of Mathematics, 831 E. 3rd St.,
Indiana University, Bloomington, IN 47405-7106} 
\thanks{%
R.L.\ is partially supported by the National
Science Foundation under grant DMS-1612363.}
\email{\href{mailto:rdlyons@indiana.edu}{rdlyons@indiana.edu}}

\title[Lower bounds for trace reconstruction]{Lower bounds for trace reconstruction}
%\date{6 Aug. 2017}
\date{\today}



\begin{abstract}
In the trace reconstruction problem, an unknown bit string  $\x\in\{0,1 \}^n$ is sent through a deletion channel where each bit is deleted independently with some probability $q\in(0,1)$, yielding a contracted string $\xt$. How many i.i.d.\ samples of $\xt$ are needed to reconstruct $\x$ with high probability? We prove that there exist $\x,\y\in\{0,1 \}^n$ such that at least $c\, n^{5/4}/\sqrt{\log n}$ traces are required to distinguish between $\x$ and $\y$ for some absolute constant $c$, improving the previous lower bound of $c\,n$. Furthermore, our result  improves the previously known lower bound for reconstruction of random strings from $c \log^2 n$ to $c \log^{9/4}n/\sqrt{\log \log n} $.
\end{abstract}

\maketitle

\bsection{Introduction}{s.intro}

In trace reconstruction, the goal is to reconstruct an unknown bit string $\x=(x_1,\dots,x_{n}) \in \cS_n:=\{0,1 \}^n$ from noisy observations of $\x$. Here we study the case where the noise is due to a deletion channel in which each bit is deleted independently with a fixed probability $q\in(0,1)$. More precisely, instead of observing $\x$, we observe many independent strings $\wt\x$ obtained by the following procedure for $k=1,\dots,n$, starting from an empty string.
\begin{itemize}
	\item (retention) With probability $p:=1-q$, copy $x_k$ to the end of $\xt$ and increase $k$ by one.
	\item (deletion) With probability $q$, only increase $k$ by one.
\end{itemize}
See Figure \ref{fig:maketrace} for an illustration. We are {\em not} given the locations of the retained bits in the original string.

\begin{figure}[ht]	
	\begin{center}
		\includegraphics[scale=1.2]{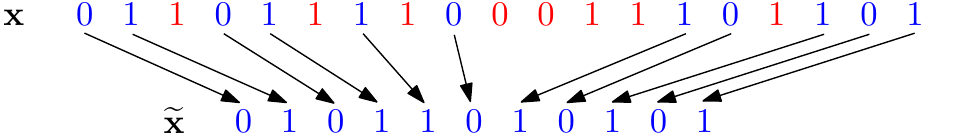}
		\caption{We obtain the trace $\xt$ by %independently 
			deleting (red) or copying (blue) each bit of $\x$. }
		\label{fig:maketrace}
	\end{center}
\end{figure}

For $T\in\N$, we consider a collection $\frk X=\{ \xt^{(1)},\dots,\xt^{(T)}
\}$ of $T$ independent outputs (called ``traces'') from the deletion channel. Our main question is the following: How many traces are needed to reconstruct $\x$ with high probability? A closely related question is, given strings $\x$ and $\y$, how many traces are needed to determine whether the input string was $\x$ or $\y$. See Section \ref{sec1} for a more precise problem statement.

\subsection{History and Results}
This problem was introduced by Batu, Kannan, Khanna, and McGregor \cite{BKKM04} as an abstraction and simplification of a fundamental problem in bioinformatics, where one desires to reconstruct a common ancestor of several organisms given genetic sequences from those organisms. Other kinds of changes can be present besides deletions, but deletions present a key difficulty. See \cite{BKKM04} for more details.

De, O'Donnell, and Servedio \cite{DOS16} and Nazarov and Peres \cite{NaPe16} prove that any string $\x\in\{ 0,1\}^n$ can be reconstructed with $\exp\left(O(n^{1/3})\right)$ traces, using the single-bit statistics of the trace. This improves the earlier upper bound of $\exp\left(n^{1/2}\op{polylog}(n)\right)$ proved by Holenstein, Mitzenmacher, Panigrahy, and Wieder \cite{HMPW08} (see \cite{MPV14} for an alternative proof). 

Previous to our paper, the best available lower bound for the number of traces needed for reconstruction was $\Omega(n)$. For example, the pair of strings $\x'_n=(0)^{n-1}1(0)^{n}\in\cS_{2n}$ and $\y'_n=(0)^{n}1(0)^{n-1}\in\cS_{2n}$ (where $(b)^m$ means a string of $m$ consecutive $b$s) requires $\Omega(n)$ traces to be distinguished (\cite[Section 4.2]{BKKM04}, \cite[Corollary 1]{MPV14}).  
Our main result is an improvement of this lower bound. Define the strings $\x_n,\y_n\in\cS_{4n}$ to be
$$
\x_n := (01)^{n-1} 10 (01)^n,\qquad \y_n := (01)^n 10 (01)^{n-1}.$$ 
These strings are periodic with period $01$ except for a single ``defect" where the period is replaced by $10$.
They can be obtained from $\x'_n$ and $\y'_n$ by replacing each $0$ with $01$ and $1$ with $10$.

\begin{theorem}\label{prop1}
	For all $q \in (0, 1)$, there is a constant $c>0$ such that for all $\eps\in(0,1/2)$ and $n \ge 2$,
	at least $T_n:=\lceil c \log(1/\eps) n^{5/4}/\sqrt{\log n}\, \rceil$ traces are required to distinguish
   between $\x_n$ and $\y_n$ with probability at least $1-\eps$. 
   In particular, $T_n$ traces are required to reconstruct all $n$-bit strings with probability at least $1-\eps$.
\end{theorem}

The following proposition is a partial converse to Theorem \ref{prop1}, and says that with $O\bigl(n^{3/2} \log n\bigr)$ traces we can distinguish between the strings $\x_n$ and $\y_n$.  
\begin{proposition}\label{prop3}
	For all $q \in (0, 1)$, there is a constant $C>0$ such that for all $\eps\in(0,1/2)$ and $n\geq 2$, $\>\lceil C\log(1/\eps)n^{3/2} \log n\rceil$ traces suffice to distinguish between $\x_n$ and $\y_n$ with probability at least $1-\eps$. 
\end{proposition}

\begin{remark}
		Building on a prior version of this paper, Zachary Chase \cite{chase-lower} recently strengthened our result by proving that Theorem \ref{prop1} still holds with $T_n:=\lceil cn^{3/2}/\log^{16}n \rceil$, where $c$ depends on $q$ and $\eps$. This means that our upper bound in Proposition \ref{prop3} is optimal up to a logarithmic factor. Using the stronger version of Theorem \ref{prop1}, Chase also improved the lower bound for random strings in Proposition \ref{prop4} below to $\lceil c\log^{5/2} n/(\log\log n)^{16}\rceil$.
\end{remark}
In general, the number of samples required to distinguish two probability measures is related to, but not determined by, the total variation distance between those measures; our appendix reviews the precise relationships. 
Given a string $\x$ and a deletion probability $q \in (0, 1)$, write $\dlch_{\x}$ for the law of the trace we obtain when applying the deletion channel with deletion probability $q$ to $\x$.  Note that the dependence on $q$ is hidden in the notation $\dlch_{\x}$.
The result of Proposition \ref{prop3} follows from the lower bound on the total variation distance $\dtv(\dlch_{\x_n},\dlch_{\y_n})$ in the following proposition:

\begin{prop}\label{TVbounds}
	For all $q \in (0, 1)$, there is a constant $C>0$ such that for all $n\ge 1$, the total variation distance $\dtv(\dlch_{\x_n},\dlch_{\y_n})$ between $\dlch_{\x_n}$ and $\dlch_{\y_n}$ satisfies
	\eqbn
	C^{-1}n^{-3/4}(\log n)^{-1/2}
	\leq 
	 \dtv(\dlch_{\x_n},\dlch_{\y_n})
	 \leq 
	 Cn^{-3/4}.
	\eqen 
\end{prop}

Above we considered reconstruction of arbitrary strings in $\cS_n$. 
The number of traces required to reconstruct an arbitrary $\x \in \cS_n$ is known as the \dfn{worst-case reconstruction problem}. We require that there exists a reconstruction algorithm such that no matter what the input string is, this string is found with high probability by the algorithm. One can also consider the \dfn{average-case reconstruction problem}. Letting $\mu_n$ denote the uniform probability measure on $\cS_n$, we assume the input string $\x$ is sampled from $\mu_n$. The question now is: 
What $T$ ensures a large probability of reconstructing $\x$?
We require that there exists a reconstruction algorithm such that if $\x\sim\mu_n$, then the algorithm identifies $\x$ with high probability when we average over both the randomness of $\x$ and the randomness of the traces. In effect, this allows us to consider only $\x\in A_n$, where $A_n\subset\cS_n$ is a set of large $\mu_n$-measure and $\cS_n\setminus A_n$ is a set of strings that are particularly difficult to reconstruct. 

Using the lower bound of $\Omega(n)$ for worst-case strings, McGregor, Price, and Vorotnikova \cite{MPV14} proved that $\Omega(\log^2n)$ traces are needed to reconstruct random strings. Following \cite{MPV14}, we state and prove a general result for transferring lower bounds for worst-case strings to lower bounds for random strings. We use this and Theorem \ref{prop1} to prove Proposition \ref{prop4}, which improves the earlier lower bound for random strings. 
\begin{prop}\label{prop4}
	For all $q \in (0, 1)$, there is a constant $c>0$ such that for all large $n$, the probability of reconstructing random $n$-bit strings from $\lceil c\log^{9/4} n/\sqrt{\log\log n}\,\rceil$ traces is at most $\exp(-n^{0.15})$.
\end{prop}
Upper bounds for  random strings are studied in \cite{BKKM04,PZ17,HPP18}. In particular, it is proved in \cite{HPP18} that $e^{O(\log^{1/3}n)}=n^{o(1)}$ traces suffice for reconstruction of random strings with any deletion probability $q\in(0,1)$.

We use the following notation throughout the paper.
\begin{notation}\label{not1}
		For two functions $f, g \colon \N \to [0, \infty)$, we write $f(n) = O\bigl(g(n)\bigr)$ if there is a constant $C>0$ such that for all sufficiently large $n$, $\>f(n) \le C\, g(n)$; $f(n) = \Omega\bigl(g(n)\bigr)$ if there is a constant $c>0$ such that for all sufficiently large $n$, $\>f(n) \ge c\, g(n)$; $f(n) = \Theta\bigl(g(n)\bigr)$ if both $f(n) = O\bigl(g(n)\bigr)$ and $f(n) = \Omega\bigl(g(n)\bigr)$; and $f(n) = o\bigl(g(n)\bigr)$ if $\lim_{n \to\infty} f(n)/g(n) = 0$. Unless otherwise specified, all constants $c,c_0,c_1,\dots,C,C_0,C_1,\dots$ and implicit constants in $\Omega(\cdot),\Theta(\cdot),O(\cdot)$ may depend on the deletion probability $q\in(0,1)$, but are independent of all other parameters.
		
      For $\x\in\cS_n$, let $\Psub_{\x}$ and $\E_{\x}$ denote probability and expectation, respectively, for the deletion channel with input string $\x$. The deletion probability is fixed and always denoted by $q$.
\end{notation}

We remark that the trace reconstruction problem has a somewhat similar flavor to the problem of reconstructing a random scenery from the observations along a random walk path \cite{bk96,mr03,mr06,mp11,hmm15}. However, to our knowledge no non-trivial lower bounds have been proved for the scenery reconstruction problem.

In the remainder of the introduction, we give a precise description of the trace reconstruction problem. We prove Theorem \ref{prop1} and the upper bound of Proposition \ref{TVbounds} in Section \ref{sec2}, Proposition \ref{prop3} and the lower bound of Proposition \ref{TVbounds} in Section \ref{sec3}, and Proposition \ref{prop4} in Section \ref{sec4}. The appendix contains some useful information about distances between probability measures and how they relate to the statistical problem of distinguishing two measures.

\smallskip
{\bf Acknowledgements: } We thank the anonymous referees for useful comments and careful reading of our paper. Most of this paper was written while visiting Microsoft Research Redmond, and we thank Microsoft for the hospitality.

\subsection{The trace reconstruction problem}
\label{sec1}
Let $\cS:=\bigcup_{n\ge 0}\cS_n$ denote the set of bit strings of finite length. Given $n \ge 0$ and $T \ge 0$, we say that (all) bit strings of length $n$ can be \dfn{reconstructed} with probability at least $1-\eps$ from $T$ traces if there is a function\footnote{Alternatively, we can replace $\cS$ by  $\bigcup_{k \le n} \cS_k$ when specifying the domain of $G$.} $G\colon\cS^T\to\{0,1 \}^n$ such that for all $\x\in\cS_n$,
\eqb
\Psubbig_{\x}{G(\mathfrak X ) = \x} \ge 1-\eps.
\label{eq1}
\eqe
If \eqref{eq1} does not hold for any choice of $G$, then we say that more than $T$ traces are required to reconstruct length-$n$ bit strings with probability $1-\eps$.

Given $n \ge 0$, $\>T \ge 0$, and $\x, \y \in \cS_n$, we say that we can \dfn{distinguish} between strings $\x$ and $\y$ with
probability at least $1 - \eps$ from $T$ traces if there is a function $G\colon\cS^T\to\{0,1 \}^n$ such that
\eqb
\begin{split}
	\Psubbig_{\x}{G(\mathfrak X ) = \x} \ge 1-\eps \quad \text{ and } \quad
	\Psubbig_{\y}{G(\mathfrak X) = \y} \ge 1-\eps.
\end{split}
\label{eq2}
\eqe
If \eqref{eq2} does not hold for any choice of $G$, then we say that more than $T$ traces are required to distinguish between $\x$ and $\y$ with probability $1-\eps$.

Recall that $\mu_n$ denotes the uniform probability measure on $\cS_n$, i.e., $\mu_n(\x)=2^{-n}$ for all $\x\in\cS_n$. We say that \dfn{random} bit strings of length $n$ can be reconstructed with probability at least $1-\eps$ from $T$ traces if there is 
a function $G\colon\cS^T\to\{0,1 \}^n$
such that 
\eqb
\sum_{\x\in\cS_n} \Psubbig_{\x}{G(\mathfrak X ) = \x}\cdot \mu_n(\x) \ge 1 - \eps.
\label{eq3}
\eqe
If \eqref{eq3} does not hold for any choice of $G$, then we say that more than $T$ traces are required to reconstruct random length $n$ bit strings with probability $1-\eps$.

Finally, we remark that one can also consider a variant of the problem where the function $G$ may be randomized, %random, 
as explained in Section \ref{sec4}, but this has no significant effect on our results.

\section{Lower bound: Proof of Theorem \ref{prop1}}
\label{sec2}
In this section, we will prove Theorem \ref{prop1}. We begin with a rough (and not entirely accurate) sketch of the proof. We will construct a coupling of the traces from $\x_n$ and $\y_n$ in two steps. The first step of the coupling is similar to what one does for $\x'_n$ and $\y'_n$, whose details can be found in \cite[Corollary 1]{MPV14}: Keep 01-blocks and 10-blocks intact, and for each 01-block decide only whether the block should be fully deleted (i.e., both bits are deleted) or not. Then the only thing we need to track is the numbers of blocks on either side of the defect that are not fully deleted. These are binomial random variables, and thus the total variation distance of the traces is at most that for binomial random variables, which is $\Theta(n^{-1/2})$. In fact, we will need to reserve some randomness, so in the first step, we delete each 01 independently with probability only $q^2/2$ (instead of $q^2$), which does not change the order of magnitude of the total variation distance.

We call the result of the first step a 2-partial trace. This is a string consisting of a sequence of 01-blocks, followed by a 10-block (i.e., the defect), followed by a sequence of 01-blocks. 
%In the second step of the coupling, we lower the bound on total variation by coupling still further on the event that the first coupling did not succeed in making the 2-partial traces the same for $\x_n$ and $\y_n$. 
Consider the event that the first coupling step did not succeed in making the 2-partial traces the same for $\x_n$ and $\y_n$. On this event, in the second step of the coupling, we increase the success probability of coupling the final traces; this gives a better bound for the total variation distance.
We do this by grouping the retained 01-blocks into 0101-blocks. Each 0101-block undergoes a deletion process that is modified because we are conditioning on the event that each of its constituent 01-blocks was not wholly deleted in the first step. By the triangle inequality, instead of coupling the 2-partial traces to each other, we may couple each to 2-partial traces with no defect. The idea is to find randomly a special 0101-block in the string without defect that becomes the same after deletion as the defect, and at the same time, has the remarkable property that what becomes of the other 0101-blocks is unaffected, so that we can couple the defect to that special 0101-block. If we can achieve that, then we use the remaining randomness to couple the numbers of 0101-blocks that are not wholly deleted in the end (these are again binomial random variables). Using a result of Liggett \cite{Lig00}, we can find that special 0101-block with high probability. Furthermore, how far the special 0101-block is from the center is controlled, which controls how far apart the binomial distributions are and leads to another factor of $O(n^{-1/4})$ in probability of failure to couple exactly. This is the most subtle part of our proof and requires careful attention to several dependencies.

Combining the two coupling stages gives that the total variation distance between the traces is $O(n^{-3/4})$. Knowing the total variation distance is not sufficient to determine the number of traces required for reconstruction (it gives a lower bound of only $\Omega(n^{3/4})$ traces; see Appendix \ref{a1}). However, by throwing away a very small set of 2-partial traces, applying Lemma \ref{prop2}, and using properties of the 2-partial traces, we can upgrade our bound on the total variation distance to show that the squared Hellinger distance between the traces is $O(n^{-5/4}\sqrt{\log n}\,)$. This yields the desired lower bound of $\Omega(n^{5/4}/\sqrt{\log n}\,)$ on the number of required traces. 

The following lemma encapsulates the overall structure in the proof of Theorem \ref{prop1}. The proof of the lemma contains almost all of our work. When we write $\dlch_{\x_n}=\mu_1+\mu_2+\mu_3$ (and the corresponding sum for $\dlch_{\y_n}$), we are adding measures as functions on points; no convolution is involved. Recall the notation $\|\cdot\|_{\ell^\infty(\cdot)}$ from \eqref{eq50}.
\begin{lemma}\label{prop10}
	For all $n \ge 2$, we have $\dlch_{\x_n}=\mu_1+\mu_2+\mu_3$ and $\dlch_{\y_n}=\mu_1+\mu'_2+\mu'_3$, where for some constant $C$ depending only on $q$,
	\eqb 
	\mu_3(\cS) = \mu'_3(\cS)\le n^{-10},
	\label{eq4}
   \vadjust{\kern3pt}
	\eqe 
	\eqb
	%\Bigl\|\frac{\mu_2(x) - \mu'_2(x)}{\mu_1(x)+\mu_2(x)}\Bigr\|_{\ell^\infty(\mu_1+\mu_2)}\leq Cn^{-1/2}\sqrt{\log n},
	\Bigl\|\frac{\mu_2 - \mu'_2}{\mu_1+\mu_2}\Bigr\|_{\ell^\infty(\mu_1+\mu_2)}\leq Cn^{-1/2}\sqrt{\log n},
	\label{eq5}
   \vadjust{\kern3pt}
	\eqe 
	\eqb
	\mu_1(x)+\mu_2(x)=0\quad\Longleftrightarrow\quad \mu_1(x)+\mu'_2(x)=0\qquad\qquad\text{for each } x\in\cS,
	\label{eq7}
	\eqe 
	and
	\eqb
	\dtv(\mu_1+\mu_2, \mu_1+\mu'_2)\le Cn^{-3/4}.
	\label{eq6}
	\eqe 
\end{lemma}

Note that \eqref{eq6} and \eqref{eq4} imply the upper bound in Proposition \ref{TVbounds}.
Before proving Lemma \ref{prop10}, we will deduce Theorem \ref{prop1} from the lemma, and we will state and prove Lemmas \ref{prop12} and \ref{prop11}, which we use in the proof of Lemma \ref{prop10}.

\begin{proof}[Proof of Theorem \ref{prop1}]
	By \rref l.Hel-sum/, \eqref{e.Hel-small}, and \eqref{eq4}, 
	\eqbn 
	\begin{split}
		\dHes( \dlch_{\x_n},\dlch_{\y_n} )
		&\leq
		\dHes(\mu_1+\mu_2, \mu_1+\mu'_2)
		+\dHes(\mu_3, \mu'_3)\\
		&\leq \dHes(\mu_1+\mu_2, \mu_1+\mu'_2)
		+ 2\,n^{-10}.
	\end{split}
	\eqen
	Let $\nu:=\mu_1+\mu_2$ and $\mu:=\mu_1+\mu'_2$.	By Lemma \ref{prop2}, \eqref{eq7}, \eqref{eq5}, and \eqref{eq6}, we get
	\eqbn \label{eq11}
	\begin{split}
		\dHes(\mu, \nu)
		&\le
		\mu\{x \st \nu(x) = 0\} +
		2 \cdot \Bigl\|\frac{\mu(x) - \nu(x)}{\nu(x)}\Bigr\|_{\ell^\infty(\nu)}
		\cdot
		\dtv(\mu, \nu)\\
		&\le 0+2Cn^{-1/2}\sqrt{\log n}\cdot  Cn^{-3/4}\\
		&= 2C^2 n^{-5/4}\sqrt{\log n}.
	\end{split}
	\eqen
	Applying \rref l.powerTV-Hel/, we obtain the theorem.
\end{proof}

Write $\Bin(n, s)$ for the binomial distribution corresponding to $n$ trials with success probability $s$ in each trial.
We record the following routine calculations for later use.

\begin{lemma} % binomial facts
	For $n \ge 1$ and $s\in(0,1)$, let $X\sim \Bin(n, s)$ and $Y\sim \Bin(n-1, s)$. Then
	\rlabel e.bin-ratio
	{
		\bigl|\P[ X=k ]-\P[Y=k]\bigr|
		=
		\frac{|ns-k|}{n(1-s)}\cdot \P[ X=k ] \quad\text{for }k=0,\dots,n
		,
	}
\rlabel e.bin-tail
{
	\Pbig{|X - ns| > c \sqrt{n \log n}\,}
   \le
   2\,n^{-2c^2} \quad\text{for } c > 0
	,
}
and
\rlabel e.bin-TV
{
	\dtv\bigl(X, Y\bigr)
	\le
	\sqrt{\frac{s}{4n(1-s)}}.
} 
\label{prop12}
\end{lemma}

\begin{proof} % binomial facts
	The equation \eqref{e.bin-ratio} follows by direct calculation:
	\eqb \label{bin-diff}
	\begin{split}
		\P[ X=k ]-\P[Y=k]
		=
		\P[ X=k ]\cdot
		\Bigl( 1 - \frac{n-k}{n(1-s)}  \Bigr)
		=
		\P[ X=k ]\cdot
		\frac{k - ns}{n(1-s)}.
	\end{split}
	\eqe
	The estimate \eqref{e.bin-tail} is immediate by the inequality of Hoeffding--Azuma. %a Chernoff bound. 
   We obtain \eqref{e.bin-TV} from \eqref{e.bin-ratio}:
	\eqbn
	\begin{split}
		\dtv(X,Y) &= \frac12 \sum_{k=0}^{n} \bigl|\P[X=k]-\P[Y=k]\bigr|
		=\frac12 \sum_{k=0}^{n} \frac{|ns-k|}{n(1-s)} \cdot \P[ X=k ]\\
		&= \frac{1}{2n(1-s)} \Ebig{|ns - X|} 
		\le \frac{\Var(X)^{1/2}}{2n(1-s)} = \sqrt{\frac{s}{4n(1-s)}}. \qedhere
	\end{split}
	\eqen
\end{proof}

The upcoming Lemma \ref{prop11} will allow us to estimate the total variation distance between traces produced from a pair of strings with and without, respectively, a defect. 
A key role in its proof is played by the following theorem of Liggett \cite{Lig00} that concerns Bernoulli processes. Part \eqref{item-2iii} of this theorem is not stated explicitly by Liggett, but follows from the proof of \cite[Proposition 2.2 and Theorem 4.25]{Lig00}.
\begin{theorem}[\cite{Lig00}]
	Let $s\in(0,1)$, and let $(a_j)_{j\in\Z}$ be a bi-infinite sequence of i.i.d.\ Bernoulli$(s)$ random variables. Then there is a random variable $X$ supported on $\N_0:=\{0,1,\dots\}$ such that the following hold.
	\begin{enumerate}[(i)]
		\item The shifted string $(b_j)_{j\in\Z}$ for $b_j:=a_{j-X}$ consists of i.i.d.\ Bernoulli$(s)$ random variables, except that $b_{0}=1$ almost surely.
		\item For a constant $C$ depending only on $s$ and for all $m\in\N$, $\>\P[X>m]\leq Cm^{-1/2}$.
		\item\label{item-2iii} Conditional on $X$ and the bits $(a_j)_{j\in\{ -X,\dots,0 \}}$, all the bits $a_j$ for $j\not\in\{ -X,\dots,0 \}$ are i.i.d.\ Bernoulli$(s)$ random variables.
	\end{enumerate}
\label{thm:lig}
\end{theorem}

Note that one \emph{cannot} choose $X$ so that $(a_j)_{j \ne -X}$ is a Bernoulli($s$) process conditioned on $X$, because that would lead to the contradiction
\[
\Ebigg{\Ebigg{\sum_{j \le 0} a_j 2^j \biggm| X}}
>
\Ebigg{\sum_{j \le 0} a_j 2^j }.
\]

We will consider strings on the alphabet $\{\sy0,\sy1,\sy2\}$. The $\sy1$s will represent 0101-blocks that become the same as the defect becomes; the $\sy0$s will represent 0101-blocks that are wholly deleted, and the $\sy2$s will represent the rest. For a string $\w=(w_1,\dots,w_n)\in\{\sy0,\sy1,\sy2 \}^n$, let $R(\w)$ denote the string obtained by deleting the $\sy0$s and then contracting the string. In other words, $R(\w)$ is obtained by repeating the following procedure for $k=1,\dots,n$, starting with an empty string:
\begin{itemize}
	\item If $w_k\in\{ \sy1,\sy2\}$, copy $w_k$ to the end of $R(\w)$ and increase $k$ by one.
	\item If $w_k=\sy0$, only increase $k$ by one.
\end{itemize}

\begin{lemma}
	\label{prop11}
	Let $C_0>1$ and $n\in\N$, and let $j_\ell,j_r\in\N$ satisfy $C_0^{-1}n< j_\ell,j_r<C_0n$. Let $\p := (p_\sy0, p_\sy1, p_\sy2) \in (0, 1)^3$ be a probability vector on the triple $(\sy0, \sy1, \sy2)$. Let $\w=(w_{-j_\ell},\dots,w_{j_r})\in\{\sy0,\sy1,\sy2 \}^{j_\ell+1+j_r}$ and 
	$\w''=(w''_{-j_\ell},\dots,w''_{j_r})\in\{\sy0,\sy1,\sy2 \}^{j_\ell+1+j_r}$ be strings of length $j_\ell+1+j_r$ on the alphabet $\{\sy0,\sy1,\sy2 \}$ such that the letters $w_i$ and $w''_i$ are i.i.d.\ with law $\p$:
	\eqb
   \w \sim \p^{j_\ell + 1 + j_r} \quad\text{ and }\quad
   \w'' \sim \p^{j_\ell + 1 + j_r} .
	\label{eq16}
	\eqe
	Condition on the event that $w_{0}=\sy1$. 
	
	Then there is a constant $C_1$ depending only on $C_0$ and $\p$ such that the total variation distance between $R(\w)$ and $R(\w'')$ is bounded above by $C_1 n^{-1/4}$.
\end{lemma}

\begin{proof} 
	Throughout the proof all constants may depend on $(C_0,p_\sy0,p_\sy1,p_\sy2)$.
	
	It will be more convenient in the proof to work with bi-infinite strings. Therefore we assume throughout the proof that $\w$ and $\w''$ are bi-infinite strings 
	$\w=(\dots,w_{-1},w_0,w_1,\dots)$ 
	and $\w''=(\dots,w''_{-1},w''_0,w''_1,\dots)$ with law $\p^{\Z}$ conditioned on the event that $w_{0}=\sy1$. We will show that the total variation distance between
	$R\big( (w_{-j_\ell},\dots,w_{j_r}) \big)$
	and
	$R\big( (w''_{-j_\ell},\dots,w''_{j_r}) \big)$
	is bounded above by $C_1 n^{-1/4}$.
	
	By the result of Liggett stated in Theorem \ref{thm:lig} above, we can find a random variable $X$ supported on $\N_0$ and independent of $\w$ and a constant $C_2>0$ (depending on $p_\sy1$) such that $\P[X\ge m]\leq C_2 (m+1)^{-1/2}$ for all $m\in\N$ and such that 
	\eqb
	(w''_{j-X})_{j\in\Z}\eqD\w.
	\label{eq17}
	\eqe
	Furthermore, by Theorem \ref{thm:lig}\eqref{item-2iii} we may define $X$ so that conditioned on $X$ and the string $(w''_{-X},\dots,w''_0)$, all letters except $w''_{-X},\dots,w''_0$ are independent with law $\p$.

\begin{figure}[ht]
	\centering
	\includegraphics[scale=1]{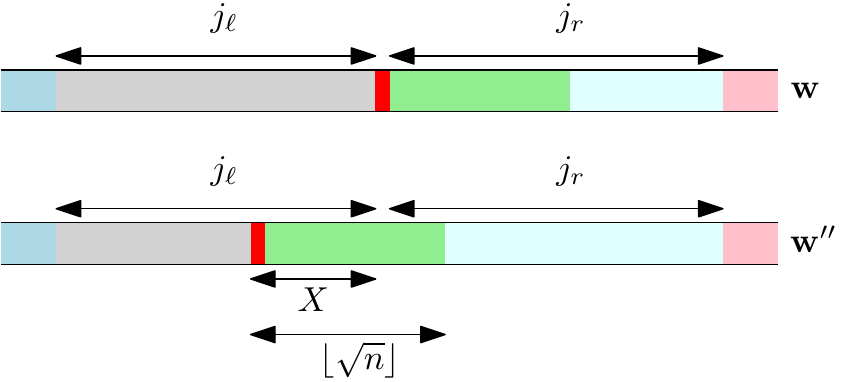}
	\caption{The figure illustrates $X$, $\w=\w_0\w_1\w_2\w_3\w_4$, and $\w''=\w''_0\w''_1\w''_2\w''_3\w''_4$ on the event $B$. The string $\w_0$ (resp., $\w_1,\w_2,\w_3,\w_4$) is shown in blue (resp., gray, green \& red, cyan, and pink), and the color code for $\w''_0,\w''_1,\w''_2,\w''_3,\w''_4$ is similar. The locations of letters known equal to $\beta$ are shown in red.   }
	\label{fig1}
\end{figure}

	Let $B:=\bigl[X<\flr{\sqrt{n}} \bigr]$, so that $\P[B\compl]\leq C_2n^{-1/4}$. On the event $B$, write the strings $\w$ and $\w''$ as concatenations of five strings each:
	\eqbn
	\w=\w_0\w_1 \w_2 \w_3\w_4\quad \text{ and } \quad
	\w''=\w''_0 \w''_1 \w''_2 \w''_3 \w''_4,
	\eqen
	where 
	\eqbn
	\begin{split}
		&\w_0  = (\dots ,w_{-j_\ell-1}),\quad
		\w_1  = (w_{-j_\ell},\dots, w_{-1}),\quad 
		\w_2 = (w_0,\dots, w_{\flr{\sqrt{n}}-1}),\\
		&\w_3 = (w_{\flr{\sqrt{n}}},\dots, w_{j_r}),\quad
		\w_4 = (w_{j_r+1},\dots),\\
		&\w''_0  = (\dots ,w''_{-j_\ell-1}),\quad
		\w''_1= (w''_{-j_\ell},\dots, w''_{-X-1}),\quad
		\w''_2 = (w''_{-X},\dots, w''_{-X+\flr{\sqrt{n}}-1}),\\
		&\w''_3 = (w''_{-X+\flr{\sqrt{n}}},\dots, w''_{j_r}),\quad
		\w''_4 = (w''_{j_r+1},\dots).
	\end{split}
	\eqen
	See Figure \ref{fig1} for an illustration. On the event $B\compl$, split the strings $\w$ and $\w''$ in the exact same way, except that 
	$\w''_0  = (\dots ,w''_{-X-1})$ and that
	$\w''_1$ is the empty string.
   By Theorem \ref{thm:lig}\eqref{item-2iii},
	\eqb
	\begin{split}
		&\text{conditional on $X$ and $\w''_2$ and on the event $B$, the letters of the strings }\\
		&\text{$\w''_0,\w''_1,\w''_3,\w''_4$ are i.i.d.\ with law $\p$.}
	\end{split}
	\label{eq:lig}
	\eqe

	Let $Y_\ell$ and $Y''_\ell$ denote the number of letters of $\w_1$ and $\w''_1$, respectively, that are \emph{not} deleted, i.e.,
	\eqbn
	Y_\ell:=\# \bigl\{ j \in\{-j_\ell,\dots,-1 \}  \st  w_j\neq \sy0 \bigr\},\qquad
	Y''_\ell:=\# \bigl\{ j \in\{-j_\ell,\dots,-X-1 \}  \st  w''_j\neq \sy0 \bigr\}.
	\eqen
	Define $Y_r$ and $Y''_r$ similarly for $\w_3$ and $\w''_3$, i.e.,
	\eqbn
	Y_r:=\# \bigl\{ j \in\{\flr{\sqrt{n}},\dots,j_r \}  \st  w_j\neq \sy0 \bigr\},\qquad
	Y''_r:=\# \bigl\{ j \in\{-X+\flr{\sqrt{n}},\dots,j_r \}  \st  w''_j\neq \sy0 \bigr\}.
	\eqen
	For any given coupling of the strings $\w$ and $\w''$, define the event $B':=\bigl[ (Y_\ell,Y_r)=(Y''_\ell,Y''_r) \bigr]$.
	
	We now define a coupling of the two strings $\w$ and $\w''$ by sampling $\w$ and $\w''$ stepwise on the same probability space as follows. Roughly, we first sample the ``central" strings $\w_2$ and $\w''_2$ so that they match, without specifying $X$. Then we sample $X$. Then, in case $B$ occurs, we sample the binomial random variables $Y_\ell$, $Y''_\ell$, $Y_r$, and $Y''_r$ so that $B'$ has as high probability as possible. Finally, we sample the rest of the information in the strings $\w$ and $\w''$. To be precise:
	\begin{enumerate}[(i)]
		\item\label{item-i} Sample $\w_2$ and $\w''_2$ such that $\w_2=\w''_2$ and the marginal law of each string is $\p^{\flr{\sqrt n}}$. This is possible by \eqref{eq17}. 
		\item\label{item-ii} Sample $X$ conditioned on $\w_2$ and $\w''_2$. (We have not described explicitly this conditional distribution; also, note that $X$ is not bounded.)
		\item\label{remaining} Sample $Y_\ell$, $Y''_\ell$, $Y_r$, and $Y''_r$ conditioned on $\w_2$, $\w''_2$, and $X$ with a special joint distribution: 
      First, $Y_\ell$ and $Y_r$ are independent, as are $Y''_\ell$ and $Y''_r$. Second,
		by \eqref{eq:lig}, conditioned on $\w_2$, $\w''_2$, and $X$, and on the event $B$, the random variables $Y''_\ell$ and $Y''_r$ are binomial random variables 
		$Y''_\ell\sim \op{Bin}(j_\ell-X,  1 - p_\sy0)$ and 
		$Y''_r   \sim \op{Bin}(j_r+X-\flr{\sqrt{n}}+1,1 - p_\sy0)$. We couple $Y_\ell$, $Y''_\ell$, $Y_r$, and $Y''_r$ such that except on an event of conditional probability $\dtv\bigl( (Y_\ell,Y_r), (Y''_\ell,Y''_r) \bigr)$ (where we consider the total variation distance conditional on $\w_2$, $\w''_2$, $X$, and $B$), the event $B'$ occurs. Third, on the event $B\compl$, we take the independent coupling of $Y_\ell$, $Y''_\ell$, $Y_r$, and $Y''_r$ conditioned on $\w_2$, $\w''_2$, and $X$.
		\item\label{item-iv} Sample the remaining randomness conditioned on $\w_2$, $\w''_2$, $X$, $Y_\ell$, $Y''_\ell$, $Y_r$, and $Y''_r$: On the event $B'\cap B$, by \eqref{eq:lig} we may couple $\w$ and $\w''$ so that $R(\w_1)=R(\w''_1)$ and $R(\w_3)=R(\w''_3)$.
	\end{enumerate}
	By \eqref{item-i} and \eqref{item-iv} of this coupling, we see that on the event $B\cap B'$,
	\eqbn
	R\big( (w_{-j_\ell},\dots,w_{j_r}) \big)
	=
	R(\w_1\w_2\w_3) = R(\w''_1\w''_2\w''_3)
	= R\big( (w''_{-j_\ell},\dots,w''_{j_r}) \big).
	\eqen
	To conclude the proof, it is therefore sufficient to show that $\P[B\cap B' ]\geq 1-C_1 n^{-1/4}$ for some constant $C_1$. 
	
By \eqref{e.bin-TV}, \eqref{eq34} (with $n = 2$ there), \eqref{remaining} of the coupling, and the fact that 
$Y_\ell\sim \op{Bin}(j_\ell, 1 - p_\sy0)$ and 
$Y_r\sim\op{Bin}(j_r-\flr{\sqrt{n}}+1, 1 - p_\sy0)$, the total variation distance between $(Y_\ell,Y_r)$ and $(Y''_\ell,Y''_r)$ conditional on $X$ and on the event $B$ is at most $C_3 X/\sqrt{n}$ for some constant $C_3$ that depends on $C_0$ and $p_\sy0$. Summing over the possible values of $X$, we get the following for some constant $C_4>0$:
	\eqbn
	\begin{split}
		\P[ B\cap (B')\compl ]
		&\leq
		C_3\E[\1_{B} X/\sqrt{n}\,] 
		= 
		C_3 n^{-1/2} \sum_{k=0}^{\flr{n^{1/2}} - 1} \P[X>k] \\
		&\leq  C_4n^{-1/2}\cdot (n^{1/2})^{1/2} = C_4n^{-1/4}.
	\end{split}
	\eqen 
	Combining this with $\P[ B\compl ]\leq C_2n^{-1/4}$, we obtain $\P[B\cap B' ]\geq 1-C_1 n^{-1/4}$ for some constant $C_1$, which concludes the proof.
\end{proof}

\begin{proof}[Proof of Lemma \ref{prop10}] 
   We will always couple the deletions made to the defects so that they are the same. If both defects are wholly deleted, then the remaining strings obviously can be coupled to have the exact same traces; this occurs with probability $q^4$ and forms part of the measure $\mu_1$ that we need to define. It will be most convenient from now on to \emph{condition on the event that neither defect is wholly deleted}.

	\begin{figure}[ht]
		\centering
		\fbox{\includegraphics[scale=1]{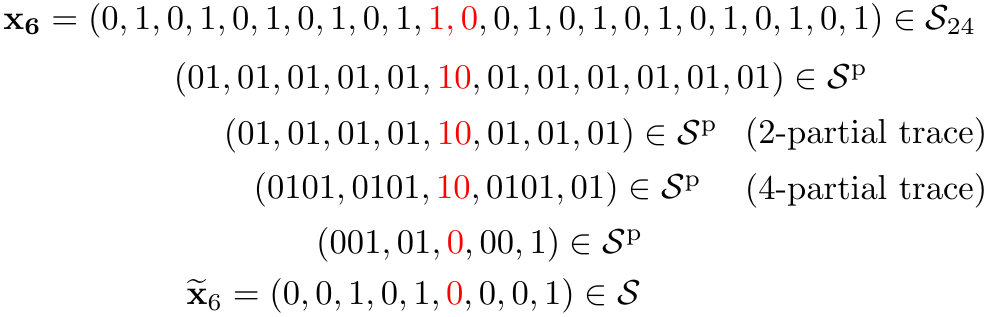}}
		\caption{The figure illustrates the steps (I)--(III) for constructing a trace as described in the proof of Lemma \ref{prop10}. The steps from the second line to the third line and from the fourth line to the fifth line are random, while the other steps are deterministic. In the displayed example, we have $Y_\ell=4$ and $Y_r=3$. The defect is colored in red to simplify the reading of the figure, but, of course, is not part of the information in the actual trace.}
		\label{partialtrace}
	\end{figure}

	A trace may be constructed in three steps (see Figure \ref{partialtrace} for an illustration):
	\begin{enumerate}[(I)]
		\item First we construct the 2-\dfn{partial trace}. A 01-\dfn{block} (resp., 10-\dfn{block}) is the string of length two given by $(0,1)$ (resp., $(1,0)$). The input string $\x_n$ may be viewed as the concatenation of $n-1$ 01-blocks, followed by a single 10-block and then $n$ 01-blocks. We sample the 2-partial trace of $\x_n$ by setting $s:=1-q^2/2$, letting $Y_\ell\sim \Bin(n-1,s)$ and $Y_r\sim\Bin(n,s)$ be independent binomial random variables, and defining the partial trace to be the concatenation of $Y_\ell$ 01-blocks, followed by a single 10-block, and then $Y_r$ 01-blocks. The partial trace of $\y_n$ is defined in the exact same way, except that $Y_\ell\sim \Bin(n,s)$ and $Y_r\sim\Bin(n-1,s)$.  
		\item Given a 2-partial trace, we define the 4-\dfn{partial trace} by the following deterministic procedure. Defining a 0101-\dfn{block} to be the length-4 string $(0,1,0,1)$, the 4-partial trace associated with the 2-partial trace in (I) is the concatenation of the following blocks in the listed order:
		\begin{itemize}
      \renewcommand\labelitemii{$\bullet$}
			\item if $Y_\ell$ is odd, a 01-block,
			\item $\lfloor Y_\ell/2\rfloor$ 0101-blocks,
			\item one 10-block (the defect),
			\item $\lfloor Y_r/2\rfloor$ 0101-blocks,
			\item if $Y_r$ is odd, a 01-block.
		\end{itemize}
		\item\label{III} From the 4-partial trace, we construct the final traces $\xt_n$ and $\yt_n$ as follows, where we treat each block independently and obtain a string in $\cS$ by concatenating the bits of the various blocks in the same order as they appear in the 4-partial trace.
		\begin{itemize}
      \renewcommand\labelitemii{$\bullet$}
			\item A 01-block is replaced by 01,1,0,$\varnothing$ with probability $p^2/s$, $pq/s$, $pq/s$, $q^2/(2s)$, respectively, where $\varnothing$ denotes the trivial (length zero) string. 
			\item A 0101-block is first replaced by two 01-blocks, and then each 01-block is treated independently as in the preceding bullet point. The result is a block in the set 
			$$
			\cB_{0101} :=\{ 0101, 101, 011, 001, 010, 01, 10,11,00,0,1,\varnothing\}.
			$$
			\item The 10-block representing the defect is replaced by 10,1,0 with probability $p^2/(1-q^2)$, $pq/(1-q^2)$, $pq/(1-q^2)$, respectively. 
		\end{itemize}
	\end{enumerate}
	Let $\cS^{\op p}$ denote the set of strings $\w=(\u_1,\dots,\u_\ell)$ for $\ell\in\N_0$, where each $\u_j$ is a 01-block, a 0101-block, or a 10-block; the case $\ell = 0$ corresponds to the empty string, $\varnothing$. In particular, both the 2-partial trace and the 4-partial trace considered above are contained in $\cS^{\op p}$. Notice that (III) provides a general procedure for obtaining a random string in $\cS$ from a string $\w\in\cS^{\op p}$; let $\Delta^{\op{p}}(\w)$ denote the resulting law on strings in $\cS$. 
	
	Let $\nu$ and $\nu'$ denote the laws of the 4-partial traces associated with $\x_n$ and $\y_n$, respectively. Then $\nu$ and $\nu'$ are probability measures on $\cS^{\op p}$. Notice that sampling the 4-partial traces above is equivalent to sampling the random variables $Y_\ell$ and $Y_r$ describing the numbers of 01-blocks on either side of the defect in the associated 2-partial trace. We decompose
	\eqb
	\nu=\nu_1+\nu_2+\nu_3 \quad \text{ and } \quad
	\nu'=\nu_1+\nu'_2+\nu'_3,
	\label{eq19}
	\eqe
	where the measures $\nu_1$, $\nu_2$, $\nu'_2$, $\nu_3$, $\nu'_3$ are defined as follows. The measures $\nu_3$ and $\nu'_3$ correspond to the events that unusually many or few 01-blocks were deleted on at least one side of the defect. More precisely, for an appropriate constant $C_0$ to be defined later, let the event $A$ be given by
	\eqb
	\begin{split}
		&A := \bigl\{ |Y_\ell-a_n| \le C_0\sqrt{n\log n},\,
		|Y_r-a_n| \le C_0\sqrt{n\log n}\, \bigr\},\quad\text{ where } a_n:=n(1-q^2/2).
	\end{split}
	\label{eq35}
	\eqe
	For an arbitrary measure $\wh\nu$ on a measurable space $\wh\cS$ and with $\wh A\subset\wh\cS$, let $\1_{\wh A}\wh\nu$ denote the measure that assigns mass $\wh\nu(U\cap \wh A)$ to any measurable set $U\subset\wh\cS$. Define
	\eqbn
	\begin{split}
		&\nu_3 := \1_{A\compl}\nu\quad\text{ and }\quad \nu'_3:=\1_{A\compl}\nu'.
	\end{split}
	\eqen
	Now choose the measures $\nu_1$, $\nu_2$, and $\nu'_2$ so that \eqref{eq19} is satisfied and $\nu_1(\cS^{\op p})$ is maximized. In particular, the measures $\nu_2$ and $\nu'_2$ have disjoint support and $\nu_2(\cS^{\op p})=\nu'_2(\cS^{\op p})=\dtv( \nu-\nu_3,\nu'-\nu'_3)$. Note that the distribution of the total number of 01-blocks is the same under $\nu_2$ as it is under $\nu'_2$.
	
	By choosing $C_0$ sufficiently large and applying  \rref e.bin-tail/, we obtain
	\eqb \label{eq20}
	\nu_3(\cS^{\op p}) = \nu'_3(\cS^{\op p})\le n^{-10}.
	\eqe
	By \rref e.bin-ratio/, for an appropriate constant $C_1$, for all $n \ge 2$, and for all $x\in\cS^{\op p}$ such that $(\nu_1+\nu_2)(x)\neq 0$, 
	\eqb \label{eq27}
	\begin{split}
		\frac{|\nu_2(x)-\nu'_2(x)|}{(\nu_1+\nu_2)(x)}
		&\leq \max_{k_1,k_2\in\{0,\dots, \lfloor C_0\sqrt{n\log n}\rfloor \}}
		\biggl| 1 - \frac{\Psubbig_{\y_n}{ |Y_\ell - a_n| = k_1,\, |Y_r-a_n| =k_2 }}
		 {\Psubbig_{\x_n}{ |Y_\ell - a_n| = k_1,\, |Y_r-a_n| =k_2 }} \biggr|\\
		&\leq C_1 n^{-1/2}\sqrt{\log n}.
	\end{split}
	\eqe
	Furthermore,
	\eqb\label{eq36}
	\nu_1(x)+\nu_2(x) = 0\quad\Longleftrightarrow\quad \nu_1(x)+\nu'_2(x) = 0\qquad\qquad\forall x\in\cS^{\op p}
	\eqe
	since each of these conditions holds if and only if the binomial random variables $Y_\ell$ and $Y_r$ associated with $x$ satisfy the condition of \eqref{eq35}. Finally, \eqref{e.bin-TV} and \eqref{eq20} give that upon increasing $C_1$ if necessary,
	\eqb
	\nu_2(\cS^{\op p}) 
   =
	\nu'_2(\cS^{\op p}) 
	= \dtv( \nu-\nu_3,\nu'-\nu'_3) 
	\leq C_1 n^{-1/2}.
	\label{eq28}
	\eqe
	
	The bounds in the preceding paragraph are expressed in term of measures on 4-partial traces. We now transfer these bounds to the final traces. Let $\wt\mu_1 := \nu_1(\cS^{\op p}) \cdot \Delta^{\op p}(\w)$ for $\w\sim \nu_1(\cS^{\op p})^{-1} \nu_1$. Define $\wt\mu_2$, $\wt\mu_3$, $\wt\mu'_2$, and $\wt\mu'_3$ similarly from the measures 
	$\nu_2$, $\nu_3$, $\nu'_2$, and $\nu'_3$. Then 
	\eqb
	\Delta_{\x_n}=\wt\mu_1+\wt\mu_2+\wt\mu_3\quad\text{ and }
	\Delta_{\y_n}=\wt\mu_1+\wt\mu'_2+\wt\mu'_3.
	\label{eq38}
	\eqe
	By \eqref{eq20} and \eqref{eq28}, we have
	\eqb 
	\wt\mu_3(\cS) = \wt\mu'_3(\cS)\le n^{-10}\quad\text{ and }\quad
	\wt\mu_2(\cS) = \wt\mu'_2(\cS) \le C_1n^{-1/2}.
	\label{eq37}
	\eqe
	Furthermore, by \eqref{eq36},
	\eqb
	\wt\mu_1(x)+\wt\mu_2(x) = 0\quad\Longleftrightarrow\quad \wt\mu_1(x)+\wt\mu'_2(x) = 0\qquad\qquad\forall x\in\cS.
	\label{eq39}
	\eqe
	Finally, by Lemma \ref{prop-pushf} and \eqref{eq27}, with $\nu_A:=\1_A\nu = \nu_1 + \nu_2$ and $\nu'_A:=\1_A\nu' = \nu'_1 + \nu'_2$, %and for all $x$ such that $(\wt\mu_1+\wt\mu_2)(x)\neq 0$, we have
	\eqb
	%\Big\|\frac{(\wt\mu_2-\wt\mu'_2)(x)}{(\wt\mu_1+\wt\mu_2)(x)}\Big\|_{\ell^\infty( \wt\mu_1+\wt\mu_2 )}
	\Big\|\frac{\wt\mu_2-\wt\mu'_2}{\wt\mu_1+\wt\mu_2}\Big\|_{\ell^\infty( \wt\mu_1+\wt\mu_2 )}
	\leq
	%\Big\|\frac{\nu_A(x)-\nu'_A(x)}{\nu_A(x)} 
	\Big\|\frac{\nu_A-\nu'_A}{\nu_A} 
	\Big\|_{\ell^\infty( \nu_A )}
	\leq
	C_1n^{-1/2}\sqrt{\log n}
	.
	\label{eq40}
	\eqe
	The preceding paragraph provides a coupling of $\Delta_{\x_n}$ and $\Delta_{\y_n}$ such that the traces are identical with probability $\wt\mu_1(\cS)>1-C_1 n^{-1/2}\sqrt{\log n}-n^{-10}$. To obtain \eqref{eq6}, we construct a better coupling by making a second attempt to couple the traces on the event that the first coupling fails. 
	
	Let 
	$$
	\sigma:=\wt\mu_2(\cS)^{-1}\wt\mu_2\quad\text{ and }\quad
	\sigma':=\wt\mu'_2(\cS)^{-1}\wt\mu'_2
	$$ 
	denote the laws of the traces on the event that the first coupling attempt failed and that the event $A$ occurs. We will argue that for an appropriate constant $C_2$,
	\eqb
	\dtv(\sigma,\sigma') \leq 2C_2n^{-1/4}.
	\label{eq21}
	\eqe
	Before proving \eqref{eq21}, we explain how \eqref{eq21} implies the lemma. 
	
	Assuming \eqref{eq21} holds, by \eqref{eq37} we can write $\wt\mu_2=\ol\mu_2+\mu_2$ and $\wt\mu'_2=\ol\mu_2+\mu'_2$, where 
	\eqb
	\mu_2(\cS) = \mu'_2(\cS) \leq 2 C_2 n^{-1/4}\cdot C_1 n^{-1/2} = 2 C_1 C_2 n^{-3/4}.
	\label{eq45}
	\eqe
	With $\mu_1:=\wt\mu_1+\ol\mu_2$, $\mu_3:=\wt\mu_3$, and $\mu'_3:=\wt\mu'_3$, all the requirements of Lemma \ref{prop10} are satisfied because of \eqref{eq38}, \eqref{eq37}, \eqref{eq39}, \eqref{eq40}, and \eqref{eq45}. Note in particular that \eqref{eq5} is satisfied because  for any $x\in\cS$ for which $\mu_1(x)+\mu_2(x)\neq 0$, 
	\eqbn
	\frac{|\mu_2(x) -\mu'_2(x)|}{\mu_1(x)+\mu_2(x)}
   =
	\frac{\bigl|\big(\wt\mu_2(x)-\ol\mu_2(x)\big) - \big(\wt\mu'_2(x)-\ol\mu_2(x)\big)\bigr|}{\big(\wt\mu_1(x)+\ol\mu_2(x)\big)+\big(\wt\mu_2(x)-\ol\mu_2(x)\big)}
	=
	\frac{|\wt\mu_2(x) -\wt\mu'_2(x)|}{\wt\mu_1(x)+\wt\mu_2(x)}
	\leq
	C_1n^{-1/2}\sqrt{\log n}.
	\eqen

	We will now prove \eqref{eq21}. Let $\w$ denote the 4-partial trace associated with $\x_n$. Let $\w''\in\cS^{\op p}$ be identical in law to $\w$ except that the 10-block (i.e., the defect) is replaced by a $0101$-block, and denote by $\sigma''$ the law of $\Delta^{\op p}(\w'')$. %Because the distributions of the lengths of $\w$ are the same whether associated to $\x_n$ or to $\y_n$, we have symmetry in $\x_n$ and $\y_n$ and may apply the triangle inequality for $\dtv$. 
	Because the lengths of the 4-partial traces associated to $\x_n$ or to $\y_n$ have the same law, we have symmetry in $\x_n$ and $\y_n$ and may apply the triangle inequality for $\dtv$. That is, it suffices to show the following in order to prove \eqref{eq21}:
	\eqb
	\dtv(\sigma,\sigma'') \leq C_2 n^{-1/4}.
	\label{eq12}
	\eqe
	When proving \eqref{eq12}, we condition on $Y_\ell$ and $Y_r$, so these random variables are viewed as constants. In particular, we will take the lengths of $\w$ and $\w''$ to be the same, i.e., the number of blocks in the two 4-partial traces is the same. We will construct a coupling of $\Delta^{\op p}(\w)$ and $\Delta^{\op p}(\w'')$ by sampling these random variables stepwise.
	
	Assume first that $Y_\ell$ and $Y_r$ are both even, namely, $2j_\ell$ and $2j_r$, respectively, for $j_\ell,j_r\in\N_0$. For each block $\u\in\cB_{0101}$, let $p_{\u}$ denote the probability that a 0101-block reduces to $\u$ in the definition of $\Delta^{\op p}$. The trace $\Delta^{\op p}(\w'')$ may be sampled in the following four steps:
	\begin{itemize}
      \item Sample the block $\ud\in\{10,1,0 \}$ that replaces the defect in $\w$, using probabilities as in the third bullet point of \eqref{III} above. 
		\item Let $(a_j)_{j\in\{ -j_\ell,\dots,j_r \}}$ be an i.i.d.\ sequence such that $a_j\in\{\sy0,\sy1,\sy2 \}$ for each $j$ and such that
		\eqb
		\P[a_j=\sy0]=p_\varnothing,\quad 
		\P[a_j=\sy1]=p_{\ud},\quad
		\P[a_j=\sy2]=1-p_\varnothing-p_{\ud}.
		\label{eq22}
		\eqe
		\item Let the $j$th block in $\w''$ reduce to $\varnothing$ (resp., $\ud$) if $a_j=\sy0$ (resp., $a_j=\sy1$).
		\item If $a_j=\sy2$, then the $j$th block in $\w''$ reduces to any given block $\u\in\cB_{0101}\setminus\{\varnothing,\ud \}$ with probability $p_{\u}/\P[a_j=\sy2]$, independently of what the other blocks reduce to.
	\end{itemize} 
	The trace $\Delta^{\op p}(\w)$ may be sampled in the exact same way, except that we condition on the event that $a_0=\sy1$. Recall the function $R$ defined preceding the statement of Lemma \ref{prop11}. Let $\rho''$ denote the law of $R\bigl((a_j)_{j\in\{ -j_\ell,\dots,j_r \}}\bigr)$, where the $a_j$s are i.i.d.\ given by \eqref{eq22}, and let $\rho$ denote $\rho''$ conditioned on $a_0=\sy1$. By \rref e.pushf-tv/,
	\eqbn
	\dtv(\sigma,\sigma'')
	\leq
	\dtv(\rho,\rho'').
	\eqen
	By Lemma \ref{prop11}, we have $\dtv(\rho,\rho'')\leq C_2n^{-1/4}$ for some constant $C_2$, which gives \eqref{eq21}.
	
	To conclude the proof, we briefly explain which modifications are needed to the above proof in the case where 
	$Y_\ell$ or $Y_r$ is odd. 
	In this case, the 4-partial traces $\w$ and $\w''$ will contain one or two 01-blocks. The total variation distance between $\Delta^{\op p}(\w)$ and $\Delta^{\op p}(\w'')$ will be the same in this case as before since we simply couple the 01-blocks of $\w$ and $\w''$ together so that they always reduce to the same block in $\{01,0,1,\varnothing \}$. 
\end{proof}

 \section{Upper bound: Proof of Proposition \ref{prop3}}
\label{sec3}
In this section we prove the lower bound of Proposition \ref{TVbounds}, which immediately implies Proposition \ref{prop3}. 

The idea in the proof of Proposition \ref{TVbounds} is to define an integer-valued random variable $Z(\xt)$ that is a function of the trace $\xt$ and such that $\dtv\bigl(Z(\xt), Z(\yt)\bigr)$ can be bounded from below. For $n\in\N$, $\x\in\cS_{4n}$, and $\xt=(\wt x_1,\dots,\wt x_\ell)$ the trace of $\x$, define $Z(\xt)$ as 
\eqb \label{defZ}
Z(\xt) := \#\bigl\{  k \st  2np+1 \le k \le (2np+\sqrt{npq})\wedge (\ell-1) , \,
\wt x_k = \wt x_{k+1} = 1
\bigr\}.
\eqe

We will use several lemmas in the proof of Proposition \ref{TVbounds}.
\begin{lemma}\label{prop7}
	We have $\Esubbig_{\y_n}{Z(\yt_n)} - \Esubbig_{\x_n}{Z(\xt_n)}=\Theta(n^{-1/2})$ and $\Esubbig_{\y_n}{Z(\yt_n)}>\Esubbig_{\x_n}{Z(\xt_n)}$ for all sufficiently large $n$.
\end{lemma}

\begin{proof}%[Proof of Lemma \ref{prop7}] 
	Let $E(j,k)$ be the event that bit $j$ in the input string is copied to position $k$ in the trace. If one or both of the positions are not well defined (i.e., if $j\not\in \{ 1,\dots,4n \}$ or if $k$ is smaller than 1 or larger than the length of the trace), then let $E(j,k)$ be the empty event. If $j, k\in \{ 1,\dots,4n \}$, then
	\eqb
	\Psubbig_{\x_n}{E(j,k)}=\Psubbig_{\y_n}{E(j,k)} = \binom{j-1}{k-1}p^{k-1}q^{j-k} \cdot p.
	\label{eq29}
	\eqe
	
	Let $\wt x_k$ (resp., $\wt y_k$) denote bit number $k$ of $\xt_n$ (resp., $\yt_n$). Assume we send the strings $\x_n$ and $\y_n$ through the deletion channel, and that the indices of the deleted bits are exactly the same for the two strings. Then the events $[\wt x_k = \wt x_{k+1} = 1]$ and $[\wt y_k = \wt y_{k+1} = 1]$ may differ only
	due to occurrence of the events $E(2n+1,k)$ or $E(2n-1,k)$ (which give $\wt y_k=1$ and $\wt x_k=1$, respectively), or
	due to occurrence of the events $E(2n+1,k+1)$ or $E(2n-1,k+1)$ (which give $\wt y_{k+1}=1$ and $\wt x_{k+1}=1$, respectively). Therefore, 
	\eqb
	\begin{split}
		&\Esubbig_{\y_n}{Z(\yt_n)} - \Esubbig_{\x_n}{Z(\xt_n)}\\		
		%%%%%%%% defect to k
		&\quad=\sum_{2np+1 \le k \leq 2np + \sqrt{npq}} \biggl(\Psubbig_{\y_n}{ E(2n+1,k)\cap\{\wt y_{k+1}=1 \} } 
		%%%%%%%% defect to k+2
		- \Psubbig_{\x_n}{ E(2n-1,k)\cap\{\wt x_{k+1}=1 \} }\\
		%%%%%%%% defect to k+1, j to k
		&\qquad\qquad\qquad+\Psubbig_{\y_n}{ \{ \wt y_k=1 \}\cap E(2n+1,k+1) } 
		%%%%%%%% defect to k+3, j to k+2
		- \Psubbig_{\x_n}{ \{\wt x_k=1 \}\cap E(2n-1,k+1) }\biggr).
	\end{split}
	\label{eq15}
	\eqe
	First we estimate the sum in \eqref{eq15} restricted to only the first two terms in each summand. %Notice that since $\x_n$ restricted to bits $\{2n-1,2n,\dots,4n-2 \}$ is identical to $\y_n$ restricted to bits $\{2n+1,2n+2,\dots,4n \}$,  
	Notice that $\x_n$ restricted to bits $\{2n-1,2n,\dots,4n-2 \}$ is identical to $\y_n$ restricted to bits $\{2n+1,2n+2,\dots,4n \}$. On the event $E(2n+1,k)$, the value of $\wt y_{k+1}$ can be obtained by sending bits $\{2n+2,2n+3,\dots,4n \}$ of $\y_n$ through the deletion channel and recording the first bit, and the analogous statement holds for $E(2n-1,k)$, $\wt x_{k+1}$, $\x_n$, and $\{2n,2n+1,\dots,4n \}$. Therefore, if $\x^{\dagger}_n$ is the string identical to $\x_n$ but with the last two bits removed, we have
	$\Psubbig_{\y_n}{\wt y_{k+1}=1 \bigm|  E(2n+1,k)}=\Psubbig_{\x^{\dagger}_n}{\wt x^{\dagger}_{k+1}=1 \bigm|  E(2n-1,k)}$.
	The probability that the bits $\{4n-1,4n\}$ of $\x_n$ affect the value of $\wt y_{k+1}$ conditioned on the event $E(2n-1,k)$ is $O(q^{2n})$. 
	Using these observations and that the considered probabilities are of order 1, we get
	\eqbn
	\frac{\Psubbig_{\y_n}{\wt y_{k+1}=1 \bigm| E(2n+1,k)}}{\Psubbig_{\x_n}{\wt x_{k+1}=1 \bigm| E(2n-1,k)}}  	
	= 1+O(q^{2n}).
	\eqen
	This and \eqref{eq29} give, with $\xi := k - 2np \in [1, \sqrt{npq}]$,
	\eqbn
	\begin{split}
		\frac{\Psubbig_{\y_n}{ E(2n+1,k)\cap\{\wt y_{k+1}=1 \} }}
		{\Psubbig_{\x_n}{ E(2n-1,k)\cap\{\wt x_{k+1}=1 \} }}
		&=
		\frac{\Psubbig_{\y_n}{ E(2n+1,k)} }
		{\Psubbig_{\x_n}{ E(2n-1,k)}}
		\cdot 
		\frac{\Psubbig_{\y_n}{\wt y_{k+1}=1 \bigm| E(2n+1,k)}}{\Psubbig_{\x_n}{\wt x_{k+1}=1 \bigm| E(2n-1,k)}}\\
		&=
		\frac{2n(2n-1)q^2}{ (2n-k)(2n-k+1) }\bigl(1+O(q^{2n})\bigr)\\
		&=
		\frac{(1-1/(2n))}{ \bigl(1-\xi/(2nq)\bigr)\bigl(1-(\xi-1)/(2nq)\bigr) }\bigl(1+O(q^{2n})\bigr)\\
		&= 1+\Theta(\xi/n),
	\end{split}
	\eqen
	and that the ratio on the left side is greater than 1 for sufficiently large $n$. %This implies
	Using this and that 
	\eqbn
	\begin{split}
	\Psubbig_{\x_n}{ E(2n-1,k)\cap\{\wt x_{k+1}=1 \} }
	&=
	\Psubbig_{\x_n}{ E(2n-1,k) }
	\cdot
	\Psubbig_{\x_n}{\wt x_{k+1}=1 \bigm| E(2n-1,k) }\\
	&=\Theta\Bigl(\frac{1}{\sqrt{n}}\Bigr)\quad 
	\mbox{when } 2np+1 \le k \leq 2np + \sqrt{npq}, % k\in\{2np+1,\dots,2np+\sqrt{npq} \},
	\end{split}
	\eqen
	we get
	\eqb
	\begin{split}
		&\sum_{2np+1 \le k \leq 2np + \sqrt{npq}}
		\Bigl(\Psubbig_{\y_n}{ E(2n+1,k)\cap\{\wt y_{k+1}=1 \} }- \Psubbig_{\x_n}{ E(2n-1,k)\cap\{\wt x_{k+1}=1 \} }\Bigr)\\
		&\qquad\qquad\qquad= 
		\Theta\biggl( 
		\sum_{2np+1 \le k \leq 2np + \sqrt{npq}}
		\Psubbig_{\x_n}{ E(2n-1,k)\cap\{\wt x_{k+1}=1 \} }
		\cdot \frac \xi n
		\biggr)=\Theta\Bigl(\frac{1}{\sqrt{n}}\Bigr),
	\end{split}
	\label{eq18}
	\eqe
	and that the left side of \eqref{eq18} is positive for all large $n$.
	
	Now we bound the sum in \eqref{eq15} restricted to only the third and the fourth term in each summand, i.e., we bound the sum
	\begin{multline}
		\sum_{2np+1 \le k \leq 2np + \sqrt{npq}}
		\Bigl(\Psubbig_{\y_n}{ \{ \wt y_k=1 \}\cap E(2n+1,k+1) } - \Psubbig_{\x_n}{ \{\wt x_k=1 \}\cap E(2n-1,k+1) }\Bigr)\\
		%%%%%%%%
		=\sum_{2np+1 \le k \leq 2np + \sqrt{npq}}\ 
		%%%%%%%% defect to k+1, j to k
		\sum_{\substack{j\leq 2n-1,\\ j \text{ odd}}}
		\Bigl(\Psubbig_{\y_n}{ E(j+1,k)\cap E(2n+1,k+1) } \\
		%%%%%%%% defect to k+3, j to k+2
		- \Psubbig_{\x_n}{ E(j-1,k)\cap E(2n-1,k+1) }\Bigr).
	\label{eq31}
	\end{multline}
	First we notice that the contribution in \eqref{eq31} from the terms for which $|j-2n|>\sqrt{npq}$ is $q^{O(\sqrt n)}$, since all the bits whose position is in $\{ j+2,\dots, 2n-2 \}$ are deleted on this event, whence
	\eqbn
	\Psubbig_{\x_n}{ E(j-1,k)\cap E(2n-1,k+1) } \leq q^{ 2n-j-3 },
	\eqen
	and a similar bound holds for $\y_n$. Therefore, in the remainder of the proof, we will consider only the terms of \eqref{eq31} for which $|j-2n|\leq \sqrt{npq}$. Notice that this condition implies $|jp-k|\leq |jp-2np|+|2np-k|\leq 2\sqrt{npq}$. By the definition of the events $E(\cdot,\cdot)$, we have for $j < 2n$,
	\eqbn
	\Psubbig_{\y_n}{  E(2n+1,k+1)\bigm| E(j+1,k) } 
	= q^{2n-j-1}p
	=
	\Psubbig_{\x_n}{  E(2n-1,k+1)\bigm| E(j-1,k) }.
	\eqen
	Using this and that $|jp-k|\leq 2\sqrt{npq}$, writing $X\sim \Bin(j-2,p)$ and $Y\sim \Bin(j,p)$, we have
	\eqb
	\begin{split}
		\frac{\Psubbig_{\y_n}{ E(j+1,k)\cap E(2n+1,k+1) }}  {\Psubbig_{\x_n}{ E(j-1,k)\cap E(2n-1,k+1) }}
		&=\frac{\Psubbig_{\y_n}{ E(j+1,k) }}{\Psubbig_{\x_n}{ E(j-1,k) }}
		= \frac{\P[ Y=k - 1 ]}{\P[ X=k - 1 ]}\\
		&= 1+O\Bigl( \frac{|jp-k|+1}{n} \Bigr)\leq 1+O\Bigl(\frac{1}{\sqrt{n}}\Bigr),
	\end{split}
	\label{eq32}
	\eqe
	where we apply \eqref{e.bin-ratio} in the second-to-last step. Furthermore, by \eqref{bin-diff} and the fact that $\E[X]<\E[Y]=jp \le (2n-1)p < 2np \le k - 1$, 
   we get $\P[ Y=k - 1 ] > \P[ X=k - 1 ]$, so
   the left side of \eqref{eq32} is greater than 1. Now we get that the right side of \eqref{eq31} is positive for large $n$ and bounded above by 
	\eqbn
	\begin{split}
		\sum_{2np+1 \le k \leq 2np + \sqrt{npq}}\ 
		\sum_{\substack{j\leq 2n-1,\\ j \text{ odd}}}
		\Psubbig_{\x_n}{ E(j-1,k)\cap E(2n-1,k+1) } \cdot O\Bigl(\frac{1}{\sqrt{n}}\Bigr)= O\Bigl(\frac{1}{\sqrt{n}}\Bigr).
	\end{split}
	\eqen
	Combining this with \eqref{eq18} gives the lemma.
\end{proof}

\begin{lemma}
	There is a constant $c>0$ depending only on $q$ such that for all $r>0$ and $n\in\N$,
	$$
	\PBig{ \bigl|Z(\yt_n)-\E[Z(\yt_n)]\bigr|>rn^{1/4} }\leq 2e^{-cr^{2}}\quad \text{ and }\quad
	\PBig{ \bigl|Z(\xt_n)-\E[Z(\xt_n)]\bigr|>rn^{1/4} }\leq 2e^{-cr^{2}}.
	$$ 
	\label{prop20}
\end{lemma}
\begin{proof}
	We will prove the result only for $\x_n$ since the proof for
	$\y_n$  is identical, and we write $Z$ instead of $Z(\x_n)$ to simplify notation. Recall from Notation \ref{not1} that all constants $c_1,c_2,\dots$ may depend on $q$ but on no other parameters. 
	
	First we prove a concentration result for a random variable $V$ that is closely related to $Z$. Let 
	$\w:=(w_1,w_2,\dots)=(01)^{\N }$ 
	be a half-infinite bit string with period $01$, and let 
	$\wt\w:=(\wt w_1,\wt w_2,\dots)$ 
	denote the trace obtained by sending $\w$ through the deletion channel with deletion probability $q$. Then set
	\eqbn
	V :=  \#\bigl\{  k \in [1, \sqrt{npq} \, ]
	\st
	\wt w_k = \wt w_{k+1} = 1
	\bigr\}.
	\eqen
	For $j\in\N$, let $u_j\sim\text{Bernoulli}(p)$ be the indicator that bit $j$ of $\w$ is not deleted. Let $E$ be the event that at least $\sqrt{npq}$ bits are \emph{not} deleted among the first $m:=\lceil 2\sqrt{npq}/p\rceil$ bits of the trace, i.e., 
	\eqbn
	E := \biggl[ \sum_{j=1}^{m} u_j \geq \sqrt{npq} \biggr].
	\eqen
	Then $\P[E\compl]\leq \exp(-c_1\sqrt{n})$ for some constant $c_1>0$ by a large-deviations bound. 
	
	Notice that $V\1_{E}$ can be written as a function of $u_1,\dots,u_m$. Furthermore, changing one $u_j$ changes $V\1_E$ by at most 2
	if both  $u_1,\dots,u_m$ and the modified sequence lie in the event $E$. By \cite{combes},
		which is a variant of McDiarmid's inequality when differences are bounded with high probability, there is a constant $c'_1>0$ such that 
	\eqbn
	\PBig{ \bigl|V\1_{E}-\E[V \mid {E}]\bigr|>rn^{1/4} } \leq 2\exp(-c'_1 r^{2} )\qquad\forall r>0.
	\eqen
   Because $\E[V \1_E]\leq \E[V \mid E] \le \E[V \1_E] + \sqrt{n p q} \exp(-c_1\sqrt{n})/(1 - e^{-c_1})$, it follows that there is a constant $c_2 > 0$ such that
	\eqb
	\PBig{ \bigl|V\1_{E}-\E[V \1_ {E}]\bigr|>rn^{1/4} } \leq 2\exp(-c_2 r^{2} )\qquad\forall r>0.
	\label{eq42}
	\eqe
	
	Now we return to the string $\x_n$. Let $u'_j$ be the indicator that the bit in position $j$ of $\x_n$ is not deleted. Let $J$ be the random variable describing the position of the bit copied to position $\lfloor 2np\rfloor $ of $\xt_n$, i.e., 
	\eqbn
	J := \inf \Big\{j\in\N\st \sum_{i=1}^{j}u'_i=\lfloor 2np\rfloor \Big \}.
	\eqen
	Extend $u'_j$ to be a Bernoulli$(p)$ process also for $j>4n$, so that $J$ is a.s.\ well defined as a natural number.
	Let $E'$ be the event that at least $\sqrt{npq}$ bits are \emph{not} deleted among the bits in position $\{ J+1,J+2,\dots,J+m \}$, i.e.,
	\eqbn
	E' := \biggl[ \sum_{j=J+1}^{J+m} u'_j \geq \sqrt{npq} \biggr].
	\eqen
	Then $\P[E']=\P[E]$. If the event  $E'':=[ J+m<4n ]$ occurs, then $V\1_{E}$ and $Z\1_{E'}$ can be coupled so they differ by at most $2$, e.g., by taking $u_j=u'_{J+j}$ for all $j$. For some constant $c_3$, $\>\Pbig{(E'')\compl}\leq \exp(-c_3 n)$. Combining these observations with the fact that $V$ and $Z$ are bounded by $\sqrt{npq}$, we obtain that for all sufficiently large $n$,
	\eqbn
	\bigl|\E[V\1_{E}]-\E[Z]\bigr| 
	\leq  
	\bigl|\E[V\1_{E}]-\E[Z\1_{E' \cap E''}]\bigr|+ \sqrt{npq}\bigl(\P[(E')\compl]+\P[(E'')\compl]\bigr) \leq 3.
	\eqen
	Assembling the above bounds, we obtain that for all sufficiently large $n$,
	\eqb
	\begin{split}
		\Pbig{ |Z-\E[Z]|>rn^{1/4}+5 }
		&\leq 
		\PBig{E' \cap E'' \cap \bigl[ |Z-\E[Z]|>rn^{1/4}+5 \bigr] } + \Pbig{(E')\compl}+\Pbig{(E'')\compl}\\
		&\leq
		\Pbig{ |V\1_E-\E[V\1_E]|>rn^{1/4} } 
		+ \exp(-c_1n^{1/2})+ \exp(-c_3 n)\\
		&\leq
		2\exp(-c_2 r^{2})+ \exp(-c_1n^{1/2})+ \exp(-c_3 n).
	\end{split}		
	\label{eq43}
	\eqe
	The first term on the right side dominates for $r=o(n^{1/4})$. Since $Z$ is bounded by $\sqrt{npq}$, the left side of \eqref{eq43} is zero for $r>\sqrt{pq}n^{1/4}$. Combining these two observations yields the lemma.
\end{proof}

\begin{lem}\label{prop21}
	Let $X$ and $Y$ be discrete, real-valued random variables such that 
	$$
   \forall r > 0 \qquad	\Pbig{|X|>r }\vee \Pbig{|Y|>r} \leq 2\exp(-r^{2}).
	$$
	Then
	\eqbn
	%\bigl|\E[X]-\E[Y]\bigr| \leq 2 \log\Bigl( \frac{2}{\dtv(X,Y)} \Bigr)\,\dtv(X,Y) + 2\, \dtv(X,Y).
	\bigl|\E[X]-\E[Y]\bigr| \leq 4\,\dtv(X,Y) \sqrt{\log \frac{2}{\dtv(X,Y)} }.
	\eqen
\end{lem}
\begin{proof}
	We let $\delta := \dtv(X,Y)$ to simplify notation. Letting $\mu_X$ and $\mu_Y$ denote the law of $X$ and $Y$, respectively, write
	\eqbn
	\mu_X = \mu+\mu^-_X+\mu^+_X\quad\text{ and }\quad
	\mu_Y = \mu+\mu^-_Y+\mu^+_Y
	,	\eqen
	where $\mu(\R)=1-\delta$, 
	$\>\mu^-_X$ and $\mu^-_Y$ are supported on $(-\infty,0)$, and 
	$\mu^+_X$ and $\mu^+_Y$ are supported on $[0,\infty)$. Then
	\eqb
	\bigl|\E[X]-\E[Y]\bigr|
	\leq
	\Big|\sum x \,\mu^-_X(x)\Big| + \Big|\sum x\, \mu^+_X(x)\Big|
	+
	\Big|\sum x\, \mu^-_Y(x)\Big| + \Big|\sum x\, \mu^+_Y(x)\Big|.
	\label{eq41}
	\eqe
   We have
	\eqbn
	\begin{split}
		\Big|\sum x \,\mu^-_X(x)\Big| + \Big|\sum x \,\mu^+_X(x)\Big|
		&=\int_{\R_+} 
		\mu^-_X\bigl((-\infty, -r]\bigr) + 
		\mu^+_X\bigl([r,\infty)\bigr)\,dr \\
		&\leq \int_{\R_+} \min\{ 2e^{-r^{2}},\delta \} \,dr 
		%&= \delta \log\Bigl( \frac{2}{\delta}\Bigr)+\delta.
		\le \delta \sqrt{\log \frac{2}{\delta}}+\frac{\delta}{2 \sqrt{\log \frac2\delta}}
      \le 2\delta \sqrt{\log \frac{2}{\delta}}.
	\end{split}
	\eqen
	% Integral bound:
	%	 \eqb
	%	 	\int_{r_0}^\infty  2e^{-r^{2}} \,dr \leq \int_{r_0}^\infty  2\frac{r}{r_0}e^{-r^2} \,dr 	=\int_{r^2_0}^\infty  \frac{1}{r_0}e^{-u} \,du	=\frac{1}{r_0}e^{-r^2_0}.
	 %	 \eqe
	Inserting this estimate and analogous estimates for $\mu^-_Y$ and $\mu^+_Y$	
	into \eqref{eq41}, we obtain the lemma,
	\eqbn
	%\bigl|\E[X]-\E[Y]\bigr| \leq 2\,\delta \log\Bigl( \frac{2}{\delta}\Bigr)+2\,\delta .
	\bigl|\E[X]-\E[Y]\bigr| \leq 4\delta \sqrt{\log \frac{2}{\delta}}.
   \qedhere
	\eqen
\end{proof}

\begin{proof}[Proof of Proposition \ref{TVbounds}, lower bound]
	By \eqref{e.pushf-tv},
	\eqb
	\dtv\bigl(Z(\xt_n),Z(\yt_n)\bigr) \leq \dtv(\dlch_{\x_n},\dlch_{\y_n}).
	\label{eq46}
	\eqe
	Letting $c$ denote the constant in Lemma \ref{prop20} and 
   \eqbn %\label{defa}
   a:=\frac 12 \bigl( \E[Z(\xt_n)]+\E[Z(\yt_n)] \bigr)=\Theta(n^{1/2}),
   \eqen
   define
	\eqbn
	X := \frac{\bigl(Z(\xt_n)-a\bigr)c}{2n^{1/4}}\quad\text{ and }\quad
	Y := \frac{\bigl(Z(\yt_n)-a\bigr)c}{2n^{1/4}}.
	\eqen 
	Notice that 
	\eqb
	\dtv(X,Y) = \dtv\bigl(Z(\xt_n),Z(\yt_n)\bigr).
	\label{eq44}
	\eqe
	By Lemma \ref{prop7},
	\eqbn
	\bigl|\E[ Z(\xt_n) ]-a\bigr|\vee \bigl|\E[ Z(\yt_n) ]-a\bigr| = \Theta(n^{-1/2}).
	\eqen
	Combining this with Lemma \ref{prop20}, we see that $X$ and $Y$ satisfy the condition of Lemma \ref{prop21} for all sufficiently large $n$. The proposition now follows from Lemma \ref{prop21}, \eqref{eq46}, \eqref{eq44}, and Lemma \ref{prop7}.
\end{proof}

\section{Lower bound for random strings: Proof of Proposition \ref{prop4}}
\label{sec4}

Recall the reconstruction problem for random strings described in Section \ref{sec1}. Proposition \ref{prop5} below transfers lower bounds for deterministic strings to lower bounds for random strings, yielding almost exponentially small success probability.  Proposition \ref{prop5} is proved by adapting the method of \cite[Theorem 1]{MPV14}. 

Proposition \ref{prop4} follows from Theorem \ref{prop1} and Proposition \ref{prop5} applied with the function $f(n) = \flr{c n^{5/4}/\sqrt{\log n}}$. The lower bound of $\Omega(\log^2 n)$ from \cite[Theorem 1]{MPV14} may be obtained from the proposition with $f(n)=\lfloor cn\rfloor$. 

In order to state the proposition, we need to describe the trace reconstruction problem with \emph{random} $G$. We say that all $n$-bit strings  can be \dfn{reconstructed with probability at least $1 - \eps$ from $T$ traces with additional randomness} if there is a Borel function $G'\colon\cS^T\times[0,1]\to\{0,1 \}^n$ 
such that for all $\x\in\cS_n$,
\eqb \label{defextra}
\int_{0}^{1}\P_{\x}[G'(\mathfrak X,t ) = \x]\,dt \ge 1-\eps\,.
\eqe

For the purpose of distinguishing between two input strings, reconstruction with extra randomness is equivalent to reconstruction without extra randomness, at least if we are willing to change $\eps$ by a factor of 2: Let $\x \ne \y$. As noted in Appendix \ref{a1}, 
\[
\min_G \Bigl(\Psubbig_\x{G(\mathfrak X) = \y} + \Psubbig_\y{G(\mathfrak X) = \x}\Bigr) = 1 - \dtv(\dlch_\x, \dlch_\y).
\]
Therefore, for any $G'\colon\cS^T\times[0,1]\to\{\x,\y\}$,
\[
\int_{0}^{1}\Bigl(\Psubbig_\x{G'(\mathfrak X,t) = \y} + \Psubbig_\y{G'(\mathfrak X,t) = \x}\Bigr)\,dt \geq 1 - \dtv(\dlch_\x, \dlch_\y).
\]
Since the maximum error probability is at most this sum of error probabilities and also is at least half the same sum, our claim follows.
In particular, the lower bound $\Omega\bigl(\log(1/\eps)n^{5/4}/\sqrt{\log n}\,\bigr)$ in Theorem \ref{prop1} also holds if we consider reconstruction with extra randomness.
A similar definition holds for reconstructing random strings with extra randomness when the random string is chosen according to a probability measure, $\rho$: one simply takes the expectation of the left side of \eqref{defextra} with $\x \sim \rho$. 

\begin{prop}\label{prop5}
	Suppose that for all $n\in\N$, the probability that all $n$-bit strings can be reconstructed with $f(n)\cdot n$ traces is at most $1 - e^{-n}$, even with extra randomness. Then for all large $n\in\N$, the probability of reconstructing random $n$-bit strings with $\flr{\frac12f(\frac12 \log n) \cdot \log n}$ traces is at most $\exp(-n^{0.15})$, even with extra randomness.
\end{prop}
\begin{proof}
	Let $r := \flr{\frac12\log n}$ and $T:=f(r)r$. 
	It was observed by Yao \cite{Y77} that von Neumann's minimax theorem yields
	\eqbn
	\min_{G'}\max_{\x\in\cS_r} \int_0^1 \Psubbig_{\x}{G'(\frk X, t) \neq \x}\,dt=\max_{\rho} \min_G \sum_{\x\in\cS_r} \Psubbig_{\x}{G(\frk X) \neq \x}\cdot\rho(\x),
	\eqen
	where we take the minima over functions $G'\colon\cS^{T}\times[0,1]\to\{0,1 \}^r$ and $G\colon\cS^{T}\to\{0,1 \}^r$, and the second maximum is over probability measures $\rho$ on $\cS_r$. By assumption, the left side is at least equal to $e^{-r}$. Therefore, there is some probability measure $\rho$ on $r$-bit strings such that $\sum_{\x\in\cS_r} \Psubbig_{\x}{G(\frk X)\neq \x}\cdot \rho(\x)\geq e^{-r}$ for
	all $G$, i.e., the probability of reconstructing an $r$-bit string 	chosen according to $\rho$ with $T$ traces is at most
	$1 - e^{-r}$. Furthermore, this result for $r$-bit strings sampled from $\rho$ holds also for reconstruction with additional randomness, since for any $\wh G\colon \cS^{T} \times[0,1] \to \{0,1 \}^r$ and $t\in[0,1]$,
	\eqbn
	\min_G \sum_{\x\in\cS_r} \Psubbig_{\x}{G(\frk X) \neq \x}\cdot\rho(\x)
	\leq 
	\sum_{\x\in\cS_r} \Psubbig_{\x}{\wh G(\frk X,t) \neq \x}\cdot\rho(\x),
	\eqen
	which implies
	\eqbn
	\min_G \sum_{\x\in\cS_r} \Psubbig_{\x}{G(\frk X) \neq \x}\cdot\rho(\x)
	\leq 
	\min_{\wh G} \sum_{\x\in\cS_r}\int_{0}^{1} \Psubbig_{\x}{\wh G(\frk X,t) \neq \x}\,dt\cdot\rho(\x).
	\eqen
	
	Sample the random uniform string $\x \in \cS_n$ in the following manner.
	Denote 
   \[
   \z_j := (x_{(j-1) r + 1}, x_{(j-1)r + 2}, \ldots, x_{j r})
   \quad\mbox{for}\quad 1 \le j \le n/r
   \]
	and $\w :=(x_{\flr{n/r}r + 1}, x_{\flr{n/r}r + 2}, \ldots, x_n)$. Write $\lambda$ for the uniform distribution on strings of length $r$ and define
	$\sigma := (\lambda - 2^{-r}\rho)/(1 - 2^{-r})$, which is a probability measure. Let $(Q_j)_{j \ge 1}$ be a Bernoulli($2^{-r}$) process.
	For each $j$, choose $\z_j$ from $\sigma$ 
	if $Q_j = 0$ and 
	from $\rho$ if $Q_j = 1$, independently for different $j$.
	Let $\w$ be uniform (independent of the preceding).
   Let $\frk X$ be the $T$ traces obtained from $\x$; it is the trace-wise concatenation of
	the traces $\frk Z_j \in \cS^T$ obtained from $\z_j$ and $\frk W \in \cS^T$ obtained from $\w$.	
	The probability of reconstructing $\x$ from $\frk X$ is at most
	the probability of reconstructing $\x$ from\footnote{We did not define this reconstruction problem, but its meaning should be obvious.} $\frk Z_1, \ldots, \frk
	Z_{\flr{n/r}r}, \frk W$ (because we could simply ignore the additional
	information in the separate traces $\frk Z_i$ and $\frk W$ that is not inherent in $\frk X$).
	Conditional on $Q_j = 1$, the probability of reconstructing $\z_j$ with $T$ traces is at most $1 - e^{-r}$ by assumption.
	Therefore, the unconditional probability of reconstructing $\z_j$ from $\frk Z_j$
	is at most $1 - 2^{-r} e^{-r}$.	Since these events are independent in $j$,
	we obtain that the probability of reconstructing $\x$ from $\frk X$ is at most $\bigl(1 - 2^{-r} e^{-r}\bigr)^{\flr{n/r}}\leq \exp( -0.9\cdot 2^{-r} e^{-r}n/r )$ for $n/r \ge 10$. Inserting the definition of $r$ gives the result. 
\end{proof}

\appendix
\section{Inequalities for distances between measures}
\label{app1}
Throughout this appendix, $\mu$ and $\nu$ are positive measures on a countable set, $X$. Most material in this appendix is standard and elementary, although we have not found a good reference presenting all the material needed for the body of our paper. One possible novelty, however, is Lemma \ref{prop2}, which we have not seen elsewhere. This lemma, though completely elementary, is a key input to the proof of Theorem \ref{prop1}, and we believe it is also useful to bound the squared Hellinger distance between two measures in many other contexts.

\subsection{Inequalities for Hellinger distance and total variation distance}
The \dfn{total variation} distance between $\mu$ and $\nu$ is defined by 
\[
\dtv(\mu, \nu)
:=
\frac12 \sum_{x \in X} \bigl|\mu(x) - \nu(x)\bigr|.
\]
Thus, in order to maximize $\sigma(X)$ over all decompositions $\mu = \sigma + \mu'$, $\nu = \sigma + \nu'$, where $\sigma$, $\mu'$, and $\nu'$ are positive measures on $X$, one takes $\sigma := \mu \wedge \nu$, yielding $\mu'(X) + \nu'(X) = 2\,\dtv(\mu, \nu)$.
If $\mu$ and $\nu$ are probability measures, then
\rlabel e.TVproba
{\dtv(\mu, \nu)
=
\max_{A \subseteq X} \bigl[\mu(A) - \nu(A)\bigr].
}
The \dfn{Hellinger} distance between $\mu$ and $\nu$ is defined by\footnote{Some authors use another normalization, e.g., with a factor of $1/\sqrt{2}$ on the right side.}
\[
\dHe(\mu, \nu)
:=
\biggl( \sum_{x \in X} \bigl[\sqrt{\mu(x)} - \sqrt{\nu(x)}\,\bigr]^2 \biggr)^{1/2}
.
\]
It is well known (e.g., \cite[Lemma 2.3]{Tsybakov}) that for probability measures $\mu$ and $\nu$, we have
\rlabel e.Hel-TV
{
\dtv(\mu, \nu)
\le
\dHe(\mu, \nu)
\le
\sqrt{2\, \dtv(\mu, \nu)}
.
}
The next lemma shows that the right-hand inequality can be strengthened if for all $x\in X$, the ratio $\mu(x)/\nu(x)$ is close to 1. Before stating it, we introduce the notation $\|f\|_{\ell^\infty(\nu)}$ for a function $f\colon X\to\R$
\eqb
\|f\|_{\ell^\infty(\nu)} := \sup\bigl\{ |f(x)|\st x\in X,\, \nu(x)\neq 0 \bigr\}.
\label{eq50}
\eqe
\begin{lemma}\label{prop2} For all positive measures $\mu$ and $\nu$, we have
	\rlabel e.Hel-infty-TV
	{
		\dHes(\mu, \nu)
		\le
		\mu\{x \st \nu(x) = 0\} +
		%2 \cdot \Bigl\|\frac{\mu(x) - \nu(x)}{\nu(x)}\Bigr\|_{\ell^\infty(\nu)}
		2 \cdot \Bigl\|\frac{\mu - \nu}{\nu}\Bigr\|_{\ell^\infty(\nu)}
		\cdot
		\dtv(\mu, \nu)
		.
	}
\end{lemma}
\begin{proof}
	Since $|a - 1| \le |a^2 - 1|$ for all $a \ge 0$, we have
	\eqaln{
		\dHes(\mu, \nu)
		- \mu\{x \st \nu(x) = 0\}
		&=
		\sum_{x \in X\st \nu(x)\neq 0} \Bigl(\sqrt{\frac{\mu(x)}{\nu(x)}} - 1\Bigr)^2 \nu(x)
		\le
		\sum_{x \in X\st \nu(x)\neq 0} \Bigl(\frac{\mu(x)}{\nu(x)} - 1\Bigr)^2 \nu(x)
		\\ &\le
		%\Bigl\|\frac{\mu(x)}{\nu(x)} - 1\Bigr\|_{\ell^\infty(\nu)}
		%\cdot \Bigl\|\frac{\mu(x)}{\nu(x)} - 1\Bigr\|_{\ell^1(\nu)}
		\Bigl\|\frac{\mu}{\nu} - 1\Bigr\|_{\ell^\infty(\nu)}
		\cdot \Bigl\|\frac{\mu}{\nu} - 1\Bigr\|_{\ell^1(\nu)}
		\,,
	}
	which is \rref e.Hel-infty-TV/.
\end{proof} 
One way to bound this $\ell^\infty$-norm is to use the following observation.
\begin{lem}
	Let $\rho$ and $\sigma$ be positive measures on a countable space $Y$, 
	let $\lambda$ be a probability measure on a measurable space $Z$, and let $\phi\colon Y\times Z\to X$ be a function. Defining $\mu := \phi_* (\rho\times \lambda)$ and $\nu := \phi_*(\sigma\times \lambda)$ to be the pushforward measures, we have
	\eqb \label{e.pushf}
		%\Bigl\|\frac{\mu(x)}{\nu(x)} - 1\Bigr\|_{\ell^\infty(\nu)}
		\Bigl\|\frac{\mu}{\nu} - 1\Bigr\|_{\ell^\infty(\nu)}
		\le
		%\Bigl\|\frac{\rho(y)}{\sigma(y)} - 1\Bigr\|_{\ell^\infty(\sigma)}.
		\Bigl\|\frac{\rho}{\sigma} - 1\Bigr\|_{\ell^\infty(\sigma)}.
	\eqe
	\label{prop-pushf}
\end{lem}
\begin{proof}
	If $\nu(x)>0$, then there must exist $y\in Y$ and $z\in Z$ such that $x=\phi(y,z)$ and $\sigma(y) > 0$. Therefore, the following holds for any $x\in X$ for which $\nu(x)>0$, with $\delta$ denoting the right side of \eqref{e.pushf} and $U\sim \lambda$:
\eqbn
\begin{split}
	|\mu(x) - \nu(x)|
	&=
	\biggl|\sum_{y\in Y\st \sigma(y)>0} \big(\rho(y)-\sigma(y)\big) \Pbig{\phi(y,U)=x}\biggr|\\
	&\leq
	\delta \sum_{y\in Y\st \sigma(y)>0} \sigma(y) \Pbig{\phi(y,U)=x}
	= \delta \cdot \nu(x).
\qedhere
\end{split}
\eqen
\end{proof}

By \rref e.TVproba/, pushing forward two probability\footnote{We remark that \rref e.pushf-tv/ also holds when $\rho$ and $\sigma$ are not probability measures.} measures by the same map cannot increase the total variation distance: 
\rlabel e.pushf-tv
{\dtv(\phi_*\rho, \phi_* \sigma) \le \dtv(\rho, \sigma).
}

The following is immediate from the definition:
\rlabel e.Hel-small
{
\dHes(\mu, \nu)
\le
\mu(X) + \nu(X)
.
}

\procl l.Hel-sum
For any positive measures $\mu_1$, $\mu_2$, $\nu_1$, and $\nu_2$ on $X$, we have
\[
\dHes(\mu_1+\mu_2, \nu_1+\nu_2)
\le
\dHes(\mu_1, \nu_1) + \dHes(\mu_2, \nu_2)
.
\]
\endprocl

\rproof
This is immediate from the inequality 
\[
\bigl(\sqrt{a+b} - \sqrt{c+d}\,\bigr)^2
\le
\bigl(\sqrt a - \sqrt c\,\bigr)^2 + \bigl(\sqrt b - \sqrt d\,\bigr)^2,\qquad a, b, c, d \ge 0.
\qedhere
\]
\Qed

The following is well known (see, e.g., \cite[page 100]{reiss89}):
\procl l.Hel-prod
For any probability measures $\mu_1, \mu_2, \nu_1, \nu_2$ on $X$, we have
\[
\dHes(\mu_1 \times \mu_2, \nu_1 \times \nu_2)
\le
\dHes(\mu_1, \nu_1) + \dHes(\mu_2, \nu_2)
.
\tag*{\qed}
\]
\endprocl

\subsection{Distinguishing between measures by independent sampling}\label{a1}
In this section, we consider two distinct probability measures $\mu$ and $\nu$, and for $m\in\N$, we consider $m$ independent samples from one of the measures. We are interested in how large we need to choose $m$ in order to determine whether our samples are from $\mu$ or $\nu$. Our bounds are expressed in terms of the Hellinger distance and the total variation distance between the measures. 

Consider first the case where $m = 1$. Let $G \colon X \to
\{\mu, \nu\}$ be a function that (roughly speaking) says whether some element $x\in X$ is more likely to be sampled from $\mu$ or $\nu$. We are interested in the sum of the
error probabilities $\mu\bigl[G(x)
= \nu\bigr] + \nu\bigl[G(x) = \mu\bigr]$. By \rref e.TVproba/,
 the error probability
sum is minimized by taking
\[
G(x)
:=
\begin{cases}
\mu &\text{if } \mu(x) \ge \nu(x),\\
\nu &\text{otherwise,}\\
\end{cases}
\]
in which case we get that the error probability sum equals $1 - \dtv(\mu, \nu)$.

Replacing $\mu,\nu$ by $\mu^m,\nu^m$ in this discussion, we get that for general $m$, the number of samples required to distinguish between $\mu$ and $\nu$ is determined precisely by $\dtv(\mu^m, \nu^m)$.

Now we derive a lower bound for the number of required samples, expressed in terms of $\dtv(\mu,\nu)$. It is well known that total variation distance can be expressed via
coupling:
\[
\dtv(\mu, \nu)
=
\min \bigl\{ \P[U \neq V] \st U \sim \mu,\, V \sim \nu \bigr\}
,
\]
where the minimum is taken over all couplings of $U$ and $V$. By using couplings of the pairs $(\mu_i, \nu_i)$ that are independent in $i$, it follows that for probability measures $\mu_1,\dots,\mu_n,\nu_1,\dots,\nu_n$,
\eqb
1 - \dtv(\mu_1 \times \cdots \times \mu_n, \nu_1 \times \cdots \times
\nu_n)
\ge
\prod_{i=1}^n \bigl[1 - \dtv(\mu_i, \nu_i)\bigr]
.
\label{eq34}
\eqe
In particular,
\rlabel e.powerTV
{
1 - \dtv(\mu^m, \nu^m)
\ge
e^{-\alpha(\mu, \nu) \cdot m\cdot \dtv(\mu, \nu)}
,
}
where 
\[
\alpha(\mu, \nu) := -\frac{\log\bigl[1-\dtv(\mu, \nu)\bigr]}{\dtv(\mu,
\nu)}.
\]
Note that $\alpha(\mu,\nu)$ approaches 1 as $\dtv(\mu, \nu) \to 0$, and that, e.g., $\alpha(\mu,\nu)$ is at most 3/2 when
$\dtv(\mu, \nu) \le 1/2$.
We can interpret \rref e.powerTV/ as saying that in order to distinguish
$\mu$ from $\nu$ when given $m$ i.i.d.\ samples from an unknown choice from
$\{\mu, \nu\}$, we need at least 
\eqb
m = \Omega\bigl(1/\dtv(\mu, \nu)\bigr)
\label{eq23}
\eqe
samples. Alternatively, we can say that if $r$ samples yield an error probability at
least $1/e$, then $r \flr{\log (1/\eps)}$ samples yield
an error probability at least $\eps$. 

Next we derive an upper bound for the number of required samples, also expressed in terms of $\dtv(\mu,\nu)$. Namely, we will prove the well-known result that we need at most $m = O\bigl(1/\dtvs(\mu,
\nu)\bigr)$ samples. By \rref e.TVproba/, we can find an event $A\subset X$ such that $\mu(A)-\nu(A)=\dtv(\mu,\nu)$. Given $m$ independent samples $\frk X_m=(x_1,\dots,x_m)$ from one of the measures, let $\ol u:=m^{-1}\sum_{j=1}^m \1_{[x_j\in A]}$ be the fraction of times that $A$ occurs. Define
 \eqbn
 G(\frk X_m) := 
 \begin{cases}
 	\mu\qquad&\text{if\,\,} \ol u>\nu(A) + \frac 12 \dtv(\mu,\nu)=\frac 12\bigl( \mu(A)+\nu(A) \bigr),\\ 
 	\nu \qquad&\text{otherwise.}
 \end{cases}
 \eqen
 An application of the inequality of Hoeffding--Azuma gives the following bound for the sum of the error probabilities:
 \eqbn
 \mu\bigl[ G(\frk X_m)=\nu \bigr] + \nu\bigl[ G(\frk X_m)=\mu \bigr]
 \leq 2\exp\bigl( -m\cdot \dtvs(\mu,\nu)/2 \bigr).
 \eqen
 In particular, 
 \eqb
 m\geq \frac{2}{\dtvs(\mu,\nu)} \log \frac{2}{\eps}
 \label{numberofsamples}
 \eqe
 samples are sufficient to distinguish between the measures with error probability at most $\eps$.
 
Both the lower and upper bounds for the number of samples required in terms of total variation distance are sharp, as illustrated by the following examples, where we use $\bern(s)$ to denote the law of a Bernoulli random variable with parameter $s\in[0,1]$: (1) $\mu := \bern(0)$ and $\nu := \bern(\delta)$, where $\dtv(\mu, \nu) = \delta$ and $\Theta(\delta^{-1})$ samples are necessary and sufficient, and (2) $\mu := \bern(1/2)$ and $\nu := \bern(1/2 + \delta)$, where $\dtv(\mu, \nu) = \delta$ and $\Theta(\delta^{-2})$ samples are necessary and sufficient. More generally, for $\alpha \in [0, 1]$, if $\mu := \bern(\delta^{1-\alpha}/2)$ and $\nu := \bern(\delta^{1-\alpha}/2 + \delta)$, then $\dtv(\mu, \nu) = \delta$ and $\Theta(\delta^{-1-\alpha})$ samples are necessary and sufficient.

If $\dHes(\mu, \nu)$ is much smaller than $\dtv(\mu, \nu)$, then the lower bound \eqref{eq23} can be improved. The following type of result seems to be folklore; we saw a version of it
in \cite[Corollary~1]{MPV14}. It says that  $m=\Omega\bigl(1/\dHes(\mu, \nu)\bigr)$ samples are necessary to distinguish between $\mu$ and $\nu$.

\procl l.powerTV-Hel
If $\mu$ and $\nu$ are probability measures with $\dHe(\mu, \nu) \le 1/2$,
then for $m \ge 1/\bigl(4\, \dHes(\mu, \nu)\bigr)$,
\[
1 - \dtv(\mu^m, \nu^m)
\ge
\exp\bigl\{- 9\,m\cdot \dHes(\mu, \nu)\bigr\}
\,.
\]
In particular, $1 - \dtv(\mu^m, \nu^m) \ge \eps$ if
\[
m
\le
\frac{1}{9\,\dHes(\mu, \nu)} \log \frac 1\eps
\,.
\]
\endprocl

\rproof
Define $r := \bflr{1/\bigl(4 \, \dHes(\mu, \nu)\bigr)} \ge 1$.
By \rref e.Hel-TV/ and \rref l.Hel-prod/,
\[
\dtvs(\mu^r, \nu^r)
\le
\dHes(\mu^r, \nu^r)
\le
r\, \dHes(\mu, \nu)
\le
1/4.
\]
This allows us to apply \rref e.powerTV/ as follows: 
\eqaln{
1 - \dtv(\mu^m, \nu^m)
&\ge
\exp\Bigl\{-\frac32 \cdot \Big\lceil\frac{m}{r}\Big\rceil \dtv(\mu^r, \nu^r)\Bigr\}
\ge
\exp\Bigl\{-3 \cdot \frac{m}{r} \dtv(\mu^r, \nu^r)\Bigr\}
\\ &\ge
\exp\Bigl\{-3 \cdot \frac{m}{r} \sqrt r\, \dHe(\mu, \nu)\Bigr\}
\ge
\exp\Bigl\{-3\sqrt{8} \cdot m \cdot\dHes(\mu, \nu)\Bigr\}
\,,
}
where in the last step, we used $r \ge 1/\bigl(8\,\dHes(\mu, \nu)\bigr)$.
\Qed

\bibliographystyle{hmralphaabbrv}
\bibliography{ref}

\begin{thebibliography}{HMPW08}

\bibitem[BK96]{bk96}
I.~Benjamini and H.~Kesten.
\newblock Distinguishing sceneries by observing the scenery along a random walk
  path.
\newblock {\em J. Anal. Math.}, 69:97--135, 1996. \MR{1428097}

\bibitem[BKKM04]{BKKM04}
T.~Batu, S.~Kannan, S.~Khanna, and A.~McGregor.
\newblock Reconstructing strings from random traces.
\newblock In {\em Proceedings of the {F}ifteenth {A}nnual {ACM}-{SIAM}
  {S}ymposium on {D}iscrete {A}lgorithms}, pages 910--918. ACM, New York, 2004.
  \MR{2290981}

\bibitem[Cha19]{chase-lower}
Z.~Chase.
\newblock New lower bounds for trace reconstruction.
\newblock {\em ArXiv e-prints}, 2019, 1905.03031.

\bibitem[Com15]{combes}
R.~Combes.
\newblock An extension of {McD}iarmid's inequality.
\newblock {\em ArXiv e-prints}, 2015, 1511.05240.

\bibitem[DOS17]{DOS16}
A.~De, R.~O'Donnell, and R.~A. Servedio.
\newblock Optimal mean-based algorithms for trace reconstruction.
\newblock In {\em S{TOC}'17---{P}roceedings of the 49th {A}nnual {ACM} {SIGACT}
  {S}ymposium on {T}heory of {C}omputing}, pages 1047--1056. ACM, New York,
  2017. \MR{3678250}

\bibitem[HMM15]{hmm15}
A.~Hart, F.~Machado, and H.~Matzinger.
\newblock Information recovery from observations by a random walk having jump
  distribution with exponential tails.
\newblock {\em Markov Process. Related Fields}, 21(4):939--970, 2015.
  \MR{3496231}

\bibitem[HMPW08]{HMPW08}
T.~Holenstein, M.~Mitzenmacher, R.~Panigrahy, and U.~Wieder.
\newblock Trace reconstruction with constant deletion probability and related
  results.
\newblock In {\em Proceedings of the {N}ineteenth {A}nnual {ACM}-{SIAM}
  {S}ymposium on {D}iscrete {A}lgorithms}, pages 389--398. ACM, New York, 2008.
  \MR{2487606}

\bibitem[HPP18]{HPP18}
N.~Holden, R.~Pemantle, and Y.~Peres.
\newblock Subpolynomial trace reconstruction for random strings \\{and}
  arbitrary deletion probability.
\newblock In S.~Bubeck, V.~Perchet, and P.~Rigollet, editors, {\em Proceedings
  of the 31st Conference On Learning Theory}, volume~75 of {\em Proceedings of
  Machine Learning Research}, pages 1799--1840. PMLR, 06--09 Jul 2018.

\bibitem[Lig02]{Lig00}
T.~M. Liggett.
\newblock Tagged particle distributions or how to choose a head at random.
\newblock In {\em In and Out of Equilibrium ({M}ambucaba, 2000)}, volume~51 of
  {\em Progr. Probab.}, pages 133--162. Birkh\"auser Boston, Boston, MA, 2002.
  \MR{1901951}

\bibitem[MP11]{mp11}
H.~Matzinger and A.~P. Pinzon.
\newblock D{NA} approach to scenery reconstruction.
\newblock {\em Stochastic Process. Appl.}, 121(11):2455--2473, 2011.
  \MR{2832409}

\bibitem[MPV14]{MPV14}
A.~McGregor, E.~Price, and S.~Vorotnikova.
\newblock Trace reconstruction revisited.
\newblock In {\em Algorithms---{ESA} 2014}, volume 8737 of {\em Lecture Notes
  in Comput. Sci.}, pages 689--700. Springer, Heidelberg, 2014. \MR{3253172}

\bibitem[MR03]{mr03}
H.~Matzinger and S.~W.~W. Rolles.
\newblock Reconstructing a piece of scenery with polynomially many
  observations.
\newblock {\em Stochastic Process. Appl.}, 107(2):289--300, 2003. \MR{1999792}

\bibitem[MR06]{mr06}
H.~Matzinger and S.~W.~W. Rolles.
\newblock Finding blocks and other patterns in a random coloring of {$\Bbb Z$}.
\newblock {\em Random Structures Algorithms}, 28(1):37--75, 2006. \MR{2187482}

\bibitem[NP17]{NaPe16}
F.~Nazarov and Y.~Peres.
\newblock Trace reconstruction with {$\exp(O(n^{1/3}))$} samples.
\newblock In {\em S{TOC}'17---{P}roceedings of the 49th {A}nnual {ACM} {SIGACT}
  {S}ymposium on {T}heory of {C}omputing}, pages 1042--1046. ACM, New York,
  2017. \MR{3678249}

\bibitem[PZ17]{PZ17}
Y.~Peres and A.~Zhai.
\newblock Average-case reconstruction for the deletion channel: subpolynomially
  many traces suffice.
\newblock In {\em 58th {A}nnual {IEEE} {S}ymposium on {F}oundations of
  {C}omputer {S}cience---{FOCS} 2017}, pages 228--239. IEEE Computer Soc., Los
  Alamitos, CA, 2017. \MR{3734232}

\bibitem[Rei89]{reiss89}
R.-D. Reiss.
\newblock {\em Approximate distributions of order statistics}.
\newblock Springer Series in Statistics. Springer-Verlag, New York, 1989.
\newblock With applications to nonparametric statistics. \MR{988164}

\bibitem[Tsy09]{Tsybakov}
A.~B. Tsybakov.
\newblock {\em Introduction to nonparametric estimation}.
\newblock Springer Series in Statistics. Springer, New York, 2009.
\newblock Revised and extended from the 2004 French original, Translated by
  Vladimir Zaiats. \MR{2724359}

\bibitem[Yao77]{Y77}
A.~C.~C. Yao.
\newblock Probabilistic computations: toward a unified measure of complexity
  (extended abstract).
\newblock In {\em 18th {A}nnual {S}ymposium on {F}oundations of {C}omputer
  {S}cience ({P}rovidence, {R}.{I}., 1977)}, pages 222--227. IEEE Comput. Sci.,
  Long Beach, Calif., 1977. \MR{0489016}

\end{thebibliography}

\end{document}